\documentclass[12pt]{amsart}
\usepackage{amsfonts,amssymb,amsmath,geometry}

\theoremstyle{plain}
\newtheorem{proposition}{Proposition}[section]
\newtheorem{example}[proposition]{Example}
\newtheorem{lemma}[proposition]{Lemma}
\newtheorem{remark}[proposition]{Remark}
\newtheorem{theorem}[proposition]{Theorem}

\numberwithin{equation}{section}

\begin{document}
\title[$q$-Fourier-Bessel series]{Uniform convergence of Fourier-Bessel
series on a q-linear grid}
\author{L. D. Abreu}
\urladdr{}
\thanks{}
\author{R. \'Alvarez-Nodarse}
\author{J.\ L.\ Cardoso}
\thanks{This paper is in final form and no version of it will be submitted
for publication elsewhere.}

\begin{abstract}
We study Fourier-Bessel series on a $q$-linear grid, defined as expansions
in complete $q$-orthogonal systems constructed with the third Jackson $q$%
-Bessel function, and obtain sufficient conditions for uniform convergence.
The convergence results are illustrated with specific examples of expansions
in $q$-Fourier-Bessel series.
\end{abstract}

\maketitle

%\date{August, 2014}

%\address{{\"y}}
%\email{{\"y}}

\section{Introduction}

Based on the orthogonality relation
\begin{equation*}
\int_{0}^{1}J_{\nu }(j_{\nu m}t)J_{\nu }(j_{\nu n}t)dt=0,
\end{equation*}%
if $\,m\neq n\,$, where $\,J_{\nu }\,$ stands for the Bessel functions of
order $\,\nu \,$ and $\,j_{\nu n}\,$ is their $\,nth\,$ positive zero, a
theory of Fourier-Bessel series was developed \cite[XVIII]{Watson}, in a
close parallelism to the classical theory of Fourier series. G. H. Hardy
\cite{Hardy} proved that, within some boundaries, the Bessel functions are
the most general functions satisfying such an orthogonality
\textquotedblleft with respect to their own zeros\textquotedblright , giving
no space for generalizations of the theory of Fourier-Bessel series, in the
scope of Lebesgue measure.

However, for a certain $q$-analogue of the Bessel function, such an
extension is possible, when considering the proper measure. The third
Jackson $q$-Bessel function $\,J_{\nu }^{(3)}(z;q)\equiv J_{\nu }(z;q)\,$(we
drop the superscript for notational convenience) or, for some authors, the
Hahn-Exton $q$-Bessel function, is defined as
\begin{equation*}
J_{\nu }^{(3)}(z;q)\equiv J_{\nu }(z;q):=z^{\nu }\frac{(q^{\nu
+1};q)_{\infty }}{(q;q)_{\infty }}\sum\limits_{k=0}^{\infty }(-1)^{k}\frac{%
q^{\frac{k(k+1)}{2}}}{(q^{\nu +1};q)_{k}(q;q)_{k}}z^{2k}\text{,}
\end{equation*}%
where$\,\nu >-1\,$ is a real parameter. When $\,q\rightarrow 1^{-}\,$we
recover the Bessel function from $\,J_{\nu }(z;q)$, after a normalization.
It is a well known fact that this function satisfies the orthogonality
relation
\begin{equation}
\int_{0}^{1}xJ_{\nu }(j_{n\nu }qx;q^{2})J_{\nu }(j_{m\nu
}qx;q^{2})d_{q}x=\eta _{n,\nu }\delta _{n,m}\text{,}  \label{ort}
\end{equation}%
\begin{equation*}
\eta _{n,\nu }=\frac{q-1}{2}q^{\nu -1}J_{\nu +1}(qj_{n\nu };q^{2})J_{\nu
}^{\prime }(j_{n\nu };q^{2})\text{,}
\end{equation*}%
where $\,j_{n\nu }(q^{2})\equiv j_{n\nu }\,$ are the positive zeros of $%
\,J_{\nu }(z;q^{2})\,$ arranged in ascending order of magnitude, $\,j_{1\nu
}<j_{2\nu }<j_{3\nu }<\cdots \,$, and $\,d_{q}x\,$ stands for the measure of
the Jackson $q$-integral.

It is our purpose in this paper to develop a theory of $q$-Fourier-Bessel
series, based on the above orthogonality relation (\ref{ort}), on results
about the completeness of these system \cite{AB}, and on the localization of
the zeros $j_{n\nu }$ \cite{ABC}. Since it was proved \cite{A2}, under the
same general conditions imposed by Hardy, that the above orthogonality
relation characterizes the functions $J_{\nu }(z;q^{2})$, this is the most
general Fourier theory based on functions $q$-orthogonal with respect to
their own zeros. Another interest in such a $q$-Fourier-Bessel theory is the
existence of a $q$-analogue of the Hankel transform with an inversion
formula, introduced in \cite{KoorS}, whose kernel is the third Jackson $q$%
-Bessel function. Such a transform was instrumental in the sampling and
Paley Wiener type theory associated with the third Jackson $q$-Bessel
function \cite{A1,AJMAA2007}.

In the papers \cite{BC}, \cite{JLC} and \cite{JLC2}, a theory of Fourier
series on a $q$-linear grid was developed, using a $q$-analogue of the
exponential function and the corresponding $\,q$-trigonometric functions
introduced by Exton \cite{E}. This was motivated by Bustoz-Suslov
orthogonality and completeness results of $q$-quadratic Fourier series \cite{BS}.
Later a simple argument has been found to prove such completeness and
results \cite{BS}, which is likely to also adapt to prove the completeness
and orthogonality of $q$-linear systems, using the expansions from \cite%
{qlinearwave}.

In this work we will first prove that pointwise convergence associated with
orthogonal discrete systems always holds and we obtain sufficient conditions
on the function for uniform convergence of the corresponding $q$%
-Fourier-Bessel series. It should be emphasized that Ismail stimulated a
considerable research activity by conjecturing properties of the zeros of $q$%
-Bessel functions, confirmed in \cite{ABC} and \cite{Hay}. First, as
documented in \cite{BH}, the asymptotic expansion for the zeros of $q$%
-difference equations has been conjectured in a letter from Ismail to
Hayman. Then, in a preprint that circulated in the early 2000's \cite{Ismail}%
, Ismail conjectured properties of the positive zeros of $q$-Bessel
functions. Several results followed, among which we can single out \cite{BH}
and \cite{Hay}, the bounds for the zeros of the third Jackson $q$-Bessel
function \cite{ABC}, the asymptotic results of \cite{B} and the recent
improvement in the corresponding accuracy \cite[Prop. A.3]{SS2016}. All
these results are contributions to the intriguing topic of functions of
order zero, whose investigation started in Littlewood's PhD thesis,
published in \cite{Littlewood}. A new interesting direction is the study of
radii of starlikeness of functions of order zero \cite%
{Starlikeness1,Starlikeness2}. The notation $J_{\nu }^{(k)},k=1,2,3$, from
\cite{I1} and is used to distinguished the three $q$-analogues of the Bessel
function defined by Jackson. As it was suggested in the very beginning of
the introduction, since the only $q$-Bessel function to appear on the text
is $J_{\nu }^{(3)}$, we will drop the superscript for shortness of the
notation and simply write %\begin{equation*}
$J_{\nu }(z;q^{2})=J_{\nu }^{(3)}(z;q^{2})$. This function shows up
naturally in the study of the quantum group of plane motions \cite{Koelink}.
We will adhere to the notations of \cite{Ism} and \cite{qcalculus}.
%\end{equation*}%
Our $q$-Fourier series will be defined in terms of orthogonal sets of the
form
\begin{equation*}
u_{n}^{(\nu )}(x)=x^{\frac{1}{2}}J_{\nu }(j_{n\nu }qx;q^{2})\text{.}
\end{equation*}%
Systems of the above form were used to obtain sampling theorems in \cite{A1}%
. In this paper we will look in more detail to the $q$-Fourier series
expansions in $L_{q}^{2}[0,1]$ associated with a function $\,f\,$
\begin{equation}
S_{q}^{\nu }[f](x)=\sum_{k=1}^{\infty }a_{k}^{\nu }\left( f\right) x^{\frac{1%
}{2}}J_{\nu }(qj_{k\nu }x;q^{2})\text{.}  \label{qseries}
\end{equation}%
Since the measure of $L_{q}^{2}[0,1]$ is discrete, pointwise convergence is
an easy consequence of the completeness results of \cite{AB,AJPA}. We will
make some comments about this in section 5 of the paper. Our main result is
the following sufficient conditions for uniform convergence of (\ref{qseries}%
). We will use the notation $V_{q}^{+}\,=\,\left\{ q^{n}:\,n=0,1,2,\ldots
\right\} $.

\begin{theorem}
\label{uniform-convergence} If the function $\,f\,$ is $\,q$-linear H\"{o}%
lder of order $\,\alpha >1\,$ in $\,V_{q}^{+}\cup \left\{ q^{-1}\right\} \,$
and such that $\,t^{-\frac{3}{2}}f(t)\in L_{q}^{2}[0,1]\,$ and the limit $\,%
\displaystyle\lim_{x\rightarrow 0^{+}}f(x)=f(0^{+})$ is finite then, the
correspondent basic Fourier-Bessel series $\,S_{q}^{(\nu )}[f](x)\,$
converges uniformly to $\,f\,$ on $\,V_{q}^{+}\,$whenever $\nu >0\,$.
\end{theorem}

The paper is organized as follows. In the next section, we collect the main
definitions and preliminary results. The third section is devoted to the
evaluation of a few finite sum. The fourth section contains a brief
introduction to the $q$-Fourier-Bessel series and the fifth section
discusses pointwise convergence for systems associated with discrete\
orthogonality relations. We prove our main result in section 6, starting
with some auxiliary Lemmas, including estimates for the coefficients of
basic Fourier-Bessel series. In the last section of the paper, two examples
of basic Fourier-Bessel expansions are provided.

\section{Definitions and preliminary results}

Following the standard notations in \cite{GR}, consider $0<q<1$, the $q$%
-shifted factorial for a finite positive integer$\,n\,$is defined as
\begin{equation*}
(a;q)_{n}=\left( 1-q\right) \left( 1-aq\right) \cdots \left(
1-aq^{n-1}\right) \,,\qquad (a;q)_{-n}=\frac{1}{\big(aq^{-n};q\big)_{n}}
\end{equation*}%
and the zero and infinite cases as%
\begin{equation*}
(a;q)_{0}=1\,,\qquad (a;q)_{\infty }=\lim\limits_{n\rightarrow \infty
}(a;q)_{n}\,\text{.}
\end{equation*}%
The symmetric $q$-difference operator acting on a suitable function $\,f\,$
is defined by
\begin{equation}
\delta _{q}f(x)=f(q^{1/2}x)-f(q^{-1/2}x)\,,
\label{symmetric-q-difference-operator}
\end{equation}%
hence, the symmetric $\,q$-derivative becomes
\begin{equation}
\frac{\delta _{q}f(x)}{\delta _{q}x}=\left\{
\begin{array}{ccc}
\displaystyle\frac{f(q^{\frac{1}{2}}x)-f(q^{-\frac{1}{2}}x)}{(q^{\frac{1}{2}
}-q^{-\frac{1}{2}})x} & \mbox{if} & x\neq 0\,, \\[1em]
f^{\prime }(0) & \mbox{if} & \;\:x=0\;\:\mbox{and}\;\:f^{\prime }(0)\;\:
\mbox{exists}\,.%
\end{array}%
\right.  \label{symmetric-q-derivative-operator}
\end{equation}%
On the opposite direction, the $q-$integral in the interval $\left(
0,a\right) $ is defined by
\begin{equation}
\int_{0}^{a}f\left( t\right) d_{q}t=\left( 1-q\right) \sum_{k=0}^{\infty
}f\left( aq^{k}\right) aq^{k}\,.  \label{qintegral}
\end{equation}%
Using this definition we may consider an inner product by setting%
\begin{equation*}
\langle f,g\rangle =\int_{0}^{1}f\left( t\right) \overline{g(t)}d_{q}t\,,
\end{equation*}%
The resulting Hilbert space is commonly denoted by $L_{q}^{2}(0,1)$.\ The
space $L_{q}^{2}(0,1)$ is a separable Hilbert space \cite{Ann}. For the
properties of the more general spaces $\,L_{q}^{p}(a,b)\,$ and $%
\,L_{q,\omega }^{p}(a,b)\,$, with $\,p\geq 1\,$, see \cite{CP}.

We will also need the following straightforward formula of $q$-integration
by parts \cite[Lemma 2, p. 5]{JLC3}, valid for $\:a,b\in\mathbb{R}\:$
assuming the involved limits exist:
\begin{equation}  \label{q-integration-by-parts}
\begin{array}{l}
\displaystyle\int_{a}^b g\big(q^{\pm\frac 1 2}x\big)\frac{\delta_q f(x)}{%
\delta_q x}d_qx\:=\: -\displaystyle\int_{a}^b f\big(q^{\mp\frac 1 2}x\big)%
\frac{\delta_q g(x)}{\delta_q x}d_qx\:+ \\[1.4em]
\hspace{1em}\displaystyle q^{\frac 1 2}\left\{\left[(fg)(bq^{-\frac 1
2})-(fg)(aq^{-\frac 1 2})\right] -\displaystyle\left[\lim_{n\rightarrow+%
\infty}(fg)(bq^{\frac 1 2+n})- \!\!\lim_{n\rightarrow+\infty}(fg)(aq^{\frac
1 2+n})\right]\right\}.%
\end{array}%
\end{equation}

The third Jackson $q$-Bessel function has a countable infinite number of
real and simple zeros, as it was shown in \cite{KS}. In \cite[Theorem 2.3]%
{ABC} it was proved that, when $\,q^{2\nu +2}<(1-q^{2})^{2}\,,$ the positive
zeros $\,j_{k\nu }=\omega _{k}^{(\nu )}(q^{2})\,$ of the function $\,J_{\nu
}(z;q^{2})\,$ satisfy
\begin{equation}
j_{k\nu }=q^{-k+\epsilon _{k}^{(\nu )}(q^{2})}  \label{zerosasymptotic}
\end{equation}%
with
\begin{equation}
0<\epsilon _{k}^{(\nu )}(q^{2})<\alpha _{k}^{(\nu )}(q^{2})\,,
\label{zerosbound1}
\end{equation}%
where
\begin{equation}
\alpha _{k}^{(\nu )}(q^{2})=\frac{\log {\left( 1-q^{2(k+\nu
)}/(1-q^{2k})\right) }}{2\log q}\,.  \label{zerosbound2}
\end{equation}%
Using Taylor expansion one proves that, as $\,k\rightarrow \infty \,,$
\begin{equation}
\alpha _{k}^{(\nu )}(q^{2})=\mathcal{O}\Big(q^{2k}\Big)\,.
\label{boundasymptotic}
\end{equation}%
%
%
%
%
%
%hence, for any constant $\,\lambda \,$, by (\ref{zerosbound1}), it follows
%that
%\begin{equation*}
%\lim_{k\rightarrow \infty }\frac{q^{-\big(k+\lambda -\epsilon _{k}^{(\nu
%)}(q)\big)^{2}}}{q^{-\big(k+\lambda\big)^2}}=1\,.
%\label{simplification}
%\end{equation*}%
Moreover \cite[Remark 2.5, page 4247]{ABC}, the above restriction on $\,q\,$
can be dropped if $\,k\,$ is chosen large enough. This is a consequence of
the fact that (\ref{zerosasymptotic})-(\ref{zerosbound2}) remain valid for
every $\,k\geq k_{0}\,$ if $\,q^{2(k_{0}+\nu )}\leq \big(1-q^{2}\big)\big(%
1-q^{2k_{0}}\big)\,.$ Hence, the following theorem holds.

\noindent\textbf{Theorem A} \emph{For every }$\,q\!\in \,]0,1[\,,$\emph{\ }$%
\,k_{0}\in N\,$\emph{\ exists such that, if }$\,k\geq k_{0}\,$\emph{\ then }%
\begin{equation*}
j_{k\nu }=q^{-k+\epsilon _{k}^{(\nu )}(q^{2})}\,,
\end{equation*}

\emph{where }$\,0<\epsilon _{k}^{(\nu )}(q^{2})<\alpha _{k}^{(\nu
)}(q^{2})\, $\emph{\ and }$\,\alpha _{k}^{(\nu )}(q^{2})\,$\emph{\ is given
by (\ref{zerosbound2}).}\

\vspace{0.7em} \noindent From now on, in order to simplify the notation, we
will consider $\,\varepsilon _{k}^{(\nu )}=\varepsilon _{k}^{(\nu )}\big(%
q^{2}\big)\,.$

\vspace{0.7em} We will refer often to the following theorem from \cite{AB}:

\noindent\textbf{Theorem B} \emph{The orthonormal sequence} $%
\,\{u_{k}\}_{k\geq 1}\,$ \emph{defined by}
\begin{equation*}
u_{k}^{(\nu)}(x)=\frac{x^{\frac{1}{2}}J_{\nu }(j_{k\nu }qx;q^{2})}{%
\left\Vert x^{\frac{1}{2}}J_{\nu }(j_{k\nu }qx;q^{2})\right\Vert }
\end{equation*}
\emph{is complete in} $\,L_{q}^{2}(0,1)$.

\vspace{0.5em} \noindent This means that, whenever a function $\,f\,$ is in $%
\,L_{q}^{2}(0,1)\,,$ if $\;\int_{0}^{1}f(x)u_{k}(x)d_{q}x=0\,,$ $%
\,k=1,2,3,\ldots \,,$ then $\,f\big(q^{k}\big)=0\,,$ $\,k=0,1,2,\ldots \,.$

The following theorem was recently proved \cite{JLC4} and will be of
fundamental importance to obtain sufficient conditions for the uniform
convergence of the basic Fourier-Bessel series.

\vspace{0.7em} \noindent\textbf{Theorem C} \emph{For large values of} $\,k\,$%
,
\begin{equation*}
\left|J_{\nu}\big(qj_{k\nu};q^2\big)\right|\leq \frac{\left(-q^2,-q^{2(%
\nu+1)};q^2\right)_{\infty}}{\left(q^2;q^2\right)_{\infty}}
q^{(k+\nu)(k-1)}\,.
\end{equation*}

\vspace{1em}

\section{\textbf{Identities for finite sums in $q$-calculus}}

In this section we intent to present some identities that will be needed in
other sections, but first, we refer the following identity
\begin{equation}  \label{symmetry}
\frac{\big(aq^{m};q\big)_k}{(a;q)_k}= \frac{\big(aq^{k};q\big)_m}{(a;q)_m}\,,
\end{equation}
valid for any $\,a\neq q^{-j}\,$, $\,j=0,1,2,\ldots\,,$ and $\,m\,$ and $%
\,k\,$ nonnegative integers, which is a consequence of the following trivial
identity
\begin{equation*}
\,(a;q)_m\big(aq^m;q\big)_k=(a;q)_k\big(aq^k;q\big)_m
\end{equation*}
which holds for every $\,a\,$ and for every integers $\,m\,$ and $\,k\,$.

\vspace{0.3em} \noindent The following two propositions and of the first
lemma can be proved by induction.

\begin{proposition}
\label{finite-sum} For $\,i=0,1,2,\cdots\,,$
\begin{equation*}
\sum_{k=0}^{i}q^{k}\frac{\big(q^{j};q\big)_k}{(q;q)_k}= \frac{\big(q^{1+j};q%
\big)_i}{(q;q)_i}\,.
\end{equation*}
\end{proposition}

%\noindent In the proof of the next proposition it is used Proposition \ref{finite-sum} with $\,q\,$ shifted to $\,q^2\,$.

\begin{proposition}
\label{lambda} For $\,\lambda=0,1,2,\cdots\,,$
\begin{equation*}
\sum_{k=0}^{i}q^{2k}\frac{\big(q^{j-1};q\big)_k}{(q;q)_k}\Big(%
q^{1+i+\lambda-k};q\Big)_1= (q;q)_1\frac{\big(q^{j+1};q\big)_i}{(q;q)_i}+ %
\big(q^{\lambda};q\big)_1q^{1+i}\frac{\big(q^{j};q\big)_i}{(q;q)_i}\,.
\end{equation*}
\end{proposition}

\noindent In the proof of the next lemma, it is used Proposition \ref{lambda}
with $\,\lambda=0\,$. %%%%%%%%%%%%%%%%%%%%%%%%%%%%%%       Caso n=1
%$$
%\sum_{k=0}^{i}q^{2k}\frac{\big(q^{j-1};q\big)_k}{(q;q)_k}\left[\frac{\big(q^{1+i-k};q\big)_1}{\big(q^{1+i};q\big)_1}+
%q\frac{\big(q^{j+i};q\big)_1}{(q^{1+i};q)_1}\frac{\big(q^{2+i-k};q\big)_1}{\big(q^{2+i};q\big)_1}\right]=
%\frac{\big(q^{1+j};q\big)_{i+1}}{(q;q)_{i+1}}\frac{\big(q^{2};q\big)_1}{\big(q^{2+i};q\big)_1}\,.
%$$

\begin{lemma}
\label{finit-sum} For each fixed non-negative integer $\,i\,$, the identity
\begin{equation*}
\sum_{k=0}^{i}q^{2k}\frac{\big(q^{j-1};q\big)_k}{(q;q)_k}\left(\sum_{%
\lambda=0}^{n}q^{\lambda} \frac{\big(q^{j+i};q\big)_{\lambda}}{\big(q^{1+i};q%
\big)_{\lambda}} \frac{\big(q^{1+i+\lambda-k};q\big)_1}{(q^{1+i+\lambda};q)_1%
}\right)= \frac{\big(q^{1+j};q\big)_{n+i}}{(q;q)_{n+i}}\frac{\big(q^{1+n};q%
\big)_1}{\big(q^{1+n+i};q\big)_1}
\end{equation*}
holds for $\,n=0,1,2,\cdots\,.$
\end{lemma}

\begin{remark}
We point out that, identity (\ref{symmetry}) enables one to rewrite the
previous results in an apparently different way: for instance, Proposition %
\ref{finite-sum} can look like
\begin{equation*}
\sum_{k=0}^{i}q^{k}\frac{\big(q^{1+k};q\big)_{j-1}}{(q;q)_{j-1}}= \frac{\big(%
q^{1+i};q\big)_j}{(q;q)_j}\,.
\end{equation*}
\end{remark}
Let us, now, consider the sum
\begin{equation*}
a_{0}^{(n,\nu )}:=q^{-n\nu }\sum_{i=0}^{n}\big(q^{2\nu }\big)^{i}\,\text{,}
\end{equation*}%
where $\,\nu \,$ is a fixed parameter. In section \ref{sec-uni-con} this
finite sum will appear in a natural way.

\noindent For $\,x\in \mathbb{R}\,$, we use the notation $\,[x]\,$ to denote the
greatest integer which does not exceed $\,x$.

\begin{lemma}
\label{product-coefficients-00} Given a sequence of numbers $\,\left\{
\gamma _{\lambda }\right\} $ then, for $\,m=0,1,2,\cdots \,,$
\begin{equation*}
\sum_{\lambda =0}^{m}a_{0}^{(\lambda ,\nu )}a_{0}^{(m-\lambda ,\nu )}\gamma
_{\lambda }=\sum_{\theta =0}^{\left[ \frac{m}{2}\right] }a_{0}^{(m-2\theta
,\nu )}\left( \sum_{\lambda =\theta }^{m-\theta }\gamma _{\lambda }\right)
\,.
\end{equation*}
\end{lemma}
\begin{proof}
Using induction, it is easy to
prove that, for $\,m=0,1,2,\cdots \,,$
\begin{equation}
a_{0}^{(\lambda ,\nu )}a_{0}^{(m-\lambda ,\nu )}=\sum_{\theta =0}^{\lambda
}a_{0}^{(m-2\theta ,\nu )}\quad \mbox{if}\qquad 0\leq \lambda \leq \left[
\frac{m}{2}\right] \,  \label{r1}
\end{equation}%
and, as a consequence,
\begin{equation}
a_{0}^{(\lambda ,\nu )}a_{0}^{(m-\lambda ,\nu )}=\sum_{\theta =0}^{m-\lambda
}a_{0}^{(m-2\theta ,\nu )}\quad \mbox{if}\qquad \left[ \frac{m}{2}\right]
+1\leq \lambda \leq m\,  \label{r2}
\end{equation}%
Then, writing
\begin{equation*}
\sum_{\lambda =0}^{m}a_{0}^{(\lambda ,\nu )}a_{0}^{(m-\lambda ,\nu )}\gamma
_{\lambda }=\sum_{\lambda =0}^{\left[ \frac{m}{2}\right] }a_{0}^{(\lambda
,\nu )}a_{0}^{(m-\lambda ,\nu )}\gamma _{\lambda }+\sum_{\lambda =\left[
\frac{m}{2}\right] +1}^{m}a_{0}^{(\lambda ,\nu )}a_{0}^{(m-\lambda ,\nu
)}\gamma _{\lambda }\text{\thinspace ,}
\end{equation*}%
using relations (\ref{r1}), (\ref{r2}) and properties of the sums, one
proves the following lemma, where, for $\,x\in \mathbb{R}\,$, $\,[x]\,$
denotes the greatest integer which do not exceed $\,x$.
\end{proof}

%%%%%%%%%%%%%%%%%%%%%%%%%%%%%%%%%%%%%%%%%%%%%%%%%%%%%%%%%%%%%%%%%%%%%%%%%%%%%%%%%%%%%%%
%%%%%%%%%%%%%%%%%%%%%%%%%%%%%%%%%%%%%%%%%%%%%%%%%%%%%%%%%%%%%%%%%%%%%%%%%%%%%%%%%%%%%%%
%%%%%%%%%%%%%%%%%%%%%%%%%%%%%%%%%%%%%%%%%%%%%%%%%%%%%%%%%%%%%%%%%%%%%%%%%%%%%%%%%%%%%%%

\section{\textbf{Fourier-Bessel Series on a }$q$-linear grid}

\noindent Using the orthogonal relation (\ref{ort}), we may consider the
Fourier Bessel series on a $q$-linear grid associated with $\,f\,$ as the sum%
\begin{equation*}
S_{q}^{\nu }[f](x)=\sum_{k=1}^{\infty }a_{k}^{\nu }\left( f\right) x^{\frac{1%
}{2}}J_{\nu }(qj_{k\nu }x;q^{2})\,,
\end{equation*}%
with the coefficients $\,a_{k}^{\nu }\,$ given by
\begin{equation*}
a_{k}^{\nu }\left( f\right) =\frac{1}{\eta _{k,\nu }}\int_{0}^{1}t^{\frac{1}{%
2}}f(t)J_{\nu }(qj_{k\nu }t;q^{2})d_{q}t\,,
\end{equation*}%
or, more conveniently,
\begin{equation}
S_{q}^{(\nu )}[f](x)=\sum_{k=1}^{\infty }a_{k}^{(\nu )}\left( f\right)
J_{\nu }(qj_{k\nu }x;q^{2})\,,  \label{fourier}
\end{equation}%
with the coefficients $\,a_{k}^{(\nu )}\,$ given by
\begin{equation}
a_{k}^{(\nu )}\left( f\right) =\frac{1}{\eta _{k,\nu }}\int_{0}^{1}tf(t)J_{%
\nu }(qj_{k\nu }t;q^{2})d_{q}t,  \label{cfourier}
\end{equation}%
where
\begin{equation}
\begin{array}{lll}
\eta _{k,\nu }=\displaystyle\int_{0}^{1}\left[ t^{\frac{1}{2}}J_{\nu }\big(%
qj_{k\nu }t;q^{2}\big)\right] ^{2}d_{q}t & = & -\frac{1-q}{2}q^{\nu
-1}J_{\nu +1}(qj_{k\nu };q^{2})J_{\nu }^{\prime }(j_{k\nu };q^{2}) \\[1em]
& = & -\displaystyle\frac{(1-q)q^{\nu -2}}{2j_{k\nu }}J_{\nu }(qj_{k\nu
};q^{2})J_{\nu }^{\prime }(j_{k\nu };q^{2})\,,%
\end{array}
\label{eta}
\end{equation}%
being the last equality valid by identity (vii)
\begin{equation}
J_{\nu }\big(qj_{k\nu };q^{2}\big)=qj_{k\nu }J_{\nu +1}\big(qj_{k\nu };q^{2}%
\big)  \label{relation-vii}
\end{equation}%
of \cite[Prop. 5, p. 8]{JLC3}.

\vspace{0.5em} \noindent With respect to the series (\ref{fourier}), our
goal is to establish pointwise convergence at each of the points $\,x \in
V_{q}^{+}=\left\{q^n:\:n=0,1,2,\cdots\right\}\,$ and to obtain sufficient
conditions for uniform convergence on $\,V_{q}^{+}$.

\section{Pointwise convergence}

\subsection{A general set-up}

\noindent With a view to study pointwise convergence of the series (\ref%
{fourier}) when $\,x\in V_{q}^{+}=\left\{ q^{n}:\,n=0,1,2,\cdots \right\} \,$%
, we first establish a general result concerning the pointwise convergence
of these series. The setting to be used in this section is a very general
one, designed to cover not only the convergence of $q$-Fourier-Bessel series
but also other Fourier systems based on discrete orthogonality relations, as
in \cite{BC}, \cite{JLC}, \cite{JLC2} and \cite{AnnMans}. There is no real
novelty in this section and we are aware that the pointwise convergence can
be extracted from the mean convergence by using known results from linear
analysis. However, we believe that the reader may benefit from the following
elegant self contained argument, which has been gently provided to us by
Professor Juan Arias de Reyna.

Let $\mathcal{N}=\{a_n\,|\, n\in{\mathbb{N}}\}$ be a numerable space and let
$\mu$ be a positive measure on $\mathcal{N}$ such that $\mu_n=\mu(\{a_n\})>0$%
. We will denote by $\mathcal{L}_{\mu }^{2}$, the space of all functions $f:%
\mathcal{N}\mapsto {\mathbb{C}}$, such that
\begin{equation*}
\Vert f\Vert _{\mathcal{L}_{\mu }^{2}}^{2}=\sum_{n=1}^{\infty
}|f(a_{n})|^{2}\mu _{n}<+\infty .
\end{equation*}%
In such a space, the scalar product $\langle f,g\rangle $ of two functions
is defined by
\begin{equation*}
\langle f,g\rangle_{\mu} =\sum_{n=1}^{\infty }f(a_{n})\overline{g(a_{n})}\mu
_{n}.
\end{equation*}
The sequence of functions $(e_{n})_{n\geq 1}$ defined on $\mathcal{N}$ by
\begin{equation}
e_{n}(a_{k})=\left\{
\begin{array}{ll}
\mu _{n}^{-1/2}, & k=n, \\[3mm]
0, & k\neq n.%
\end{array}%
\right.  \label{def-e_k}
\end{equation}
is a complete orthonormal system in $\mathcal{L}_{\mu }^{2}$. To check this
fact, notice that the function $g_{N}$, $N\in{\mathbb{N}}$, defined by
\begin{equation*}
g_{N}=f-\sum_{n=1}^{N}\langle f,e_{n}\rangle_{\mu} e_{n}\,,\qquad f\in
\mathcal{L}_{\mu}^{2}\,,
\end{equation*}%
is such that $g_{N}(a_{k})=0$ for all $k\leq N$ and $g_{N}(a_{k})=f(a_{k})$
for all $k>N$. Therefore,
\begin{equation*}
\Vert g_{N}\Vert _{\mathcal{L}_{\mu }^{2}}^{2}=\sum_{n=N+1}^{\infty
}|f(a_{n})|^{2}\mu _{n}\rightarrow 0,\,\,\,\mbox{as}\,\,\,N\rightarrow
\infty .
\end{equation*}%
Thus, for an arbitrary $f\in \mathcal{L}_{\mu }^{2}$ , we have
\begin{equation*}
f=\sum_{n=1}^{\infty }\langle f,e_{n}\rangle_{\mu} e_{n},
\end{equation*}
with convergence in norm $\Vert \cdot \Vert _{\mathcal{L}_{\mu }^{2}}^{2}$ .
This is also true for any other complete orthonormal system $(u_{n})_{n\geq
1}$, i.e., for an arbitrary $f\in \mathcal{L}_{\mu }^{2}$ one has
\begin{equation*}
f=\sum_{n=1}^{\infty }\langle f,u_{n}\rangle_{\mu} u_{n},
\end{equation*}
with convergence in norm $\Vert \cdot \Vert _{\mathcal{L}_{\mu }^{2}}^{2}$.
It remains only to check when the convergence of the above series is
pointwise. The answer to this question is in the following lemma.

\begin{lemma}
\label{juan} Let $(u_n)_{n\geq 1}$ be a complete orthonormal system in $%
\mathcal{L}^2_{\mu}$. Then for any arbitrary $f\in\mathcal{L}^2_{\mu}$
\begin{equation*}
f(a_k)=\lim_{N\to\infty}\sum_{n=1}^N \langle f, u_n\rangle u_n(a_k), \quad
\forall\,a_k\in\mathcal{N}.
\end{equation*}
\end{lemma}

\begin{proof}
Let $a_k$ be an arbitrary element of $\mathcal{N}$. Then, the function $%
d_k:=\mu_k^{-1/2}e_k$, where $e_k$ is the function given in (\ref{def-e_k}),
satisfies the property
\begin{equation*}
\langle f, d_k\rangle=\langle f, \mu_k^{-1/2}e_k\rangle=f(a_k)\mu_k^{-1/2}
e_k(a_k)\mu_k=f(a_k).
\end{equation*}
In particular, $\langle u_n, d_k \rangle=u_n(a_k)$. Then,
\begin{equation*}
f(a_k)=\langle f, d_k\rangle=\left\langle
\lim_{N\to\infty}\sum_{n=1}^N\langle f, u_n\rangle u_n, d_k \right\rangle=
\lim_{N\to\infty} \sum_{n=1}^N\langle f, u_n\rangle \langle u_n, d_k \rangle,
\end{equation*}
and, therefore,
\begin{equation*}
f(a_k)=\lim_{N\to\infty} \sum_{n=1}^N \langle f, u_n\rangle u_n(a_k).
\end{equation*}
\end{proof}

\subsection{Application to $q$-Fourier-Bessel series}

Let be $\,\mathcal{N}=V_{q}^{+}\,$ and in $\,V_{q}^{+}\,$ define the measure
$\,\mu\,$ associated to the Jackson $q$-integral (\ref{qintegral}). Let $%
\,L_{q}^{2}[0,1]\,$ be the corresponding $\,\mathcal{L}^{2}_\mu\,$ space.
Since the set of functions
\begin{equation*}
u_{n}^{(\nu)}(x)=\frac{x^{\frac{1}{2}}J_{\nu }(j_{n\nu }qx;q^{2})}{
\left\Vert x^{\frac{1}{2}}J_{\nu }(j_{n\nu
}qx;q^{2})\right\Vert_{L_{q}^{2}[0,1]}},
\end{equation*}
is a complete orthonormal system in $L_{q}^{2}[0,1]$, then, for an arbitrary
$f\in L_{q}^{2}[0,1]$, i.e., $f$ such that
\begin{equation*}
\int_{0}^{1}|f(x)|^{2}d_{q}x<+\infty ,
\end{equation*}%
we have the equality%
\begin{equation*}
f(q^{k})=\lim_{N\rightarrow \infty }\sum_{n=1}^{N}\langle f,u_{n}^{(\nu)
}\rangle u_{n}^{(\nu)}(q^{k}),\quad k=0,1,2,\ldots ,
\end{equation*}%
where
\begin{equation*}
\langle f,u_{n}^{(\nu)}\rangle =\int_{0}^{1}f(t)u_{n}^{(\nu)}(t)d_{q}t\,.
\end{equation*}%
This summarizes in the following theorem.

\begin{theorem}
If $f\in L_{q}^{2}[0,1]$, then the $q$-Fourier-Bessel series (\ref{fourier})
converges to the function $f$ at every point $x\in V_q^{+}$.
\end{theorem}

\begin{remark}
Let us mention that in the case of the standard trigonometric series the
equivalent result of Lemma \ref{juan} ($\mathcal{L}^2_{\mu}$ convergence
implies pointwise convergence) is not true. In fact this problem leads to
the celebrated Carleson Theorem (see e.g. \cite{Aria}). The main difference
between these two cases is that, contrary to the case of the discrete space $%
\mathcal{L}^2_\mu$ (see the function $d_k$ used in the proof of Lemma \ref%
{juan}), for functions $f\in\mathcal{L}^2([0,2\pi])$ and for every $a\in[%
0,2\pi]$, there not exists functions $f_a$ such that $\langle f, f_a\rangle
=f(a)$.
\end{remark}

\begin{remark}
In \cite{JLC}, some convergence theorems of $q$-Fourier series associated
with the $q$-trigonometric orthogonal system $\Big\{1,\,C_{q}\big(q^{\frac{1%
}{2}}\omega _{k}x\big),\,S_{q}\left( q\omega _{k}x\right) \Big\}$ were
established, where the $q$-cosines $C_{q}$ and $q$-sinus $S_{q}$ can be
defined in terms of the third $q$-Bessel functions by the identities
\begin{equation*}
C_{q}(z)=q^{-3/8}\frac{(q^{2};q^{2})_{\infty }}{(q;q^{2})_{\infty }}%
z^{1/2}J_{-1/2}\big(q^{-3/4}z;q^{2}\big),\qquad S_{q}(z)=q^{1/8}\frac{%
(q^{2};q^{2})_{\infty }}{(q;q^{2})_{\infty }}z^{1/2}J_{1/2}\big(%
q^{-1/4}z;q^{2}\big)\;,
\end{equation*}%
being $\left\{ \omega _{k}\right\} $ the sequence of positive zeros of the
function $S_{q}$, arranged in ascendant order of magnitude. Notice that,
since this orthogonal system is a complete system (\cite{BC}) in $%
L_{q}^{2}[-1,1]$, then one can derive in a similar way that the $q$%
-trigonometric Fourier series converge to $\,f\in L_{q}^{2}[-1,1]\,$ at
every point of $\,V_{q}\,=\,\left\{ \pm q^{n}:\,n=0,1,2,\ldots \right\} $,
i.e., for every fixed $\,x\in V_{q}\,$,
\begin{equation*}
f(x)=\frac{a_{0}}{2}+\sum_{k=1}^{\infty }\left\{ a_{k}C_{q}\big(q^{\frac{1}{2%
}}\omega _{k}x\big)+b_{k}S_{q}\big(q\omega _{k}x\big)\right\} ,
\end{equation*}%
with $a_{0}=\int_{-1}^{1}f(t)d_{q}t$ and, for $k=1,2,3,\ldots ,$
\begin{equation*}
a_{k}=\frac{1}{\tau _{k}}\int_{-1}^{1}f(t)C_{q}\big(q^{\frac{1}{2}}\omega
_{k}t\big)d_{q}t,\quad b_{k}=\frac{q^{\frac{1}{2}}}{\tau _{k}}%
\int_{-1}^{1}f(t)S_{q}\left( q\omega _{k}t\right) d_{q}t,
\end{equation*}%
where,
\begin{equation*}
\tau _{k}=(1-q)C_{q}(q^{1/2}\omega _{k})S_{q}^{\prime }(\omega _{k})\,.\,
\end{equation*}%
Thus, the corresponding open problem posed in the concluding remarks section
of \cite{JLC} is completely solved.
\end{remark}

\begin{remark}
In \cite{AnnMans} a rigorous theory of $q$-Sturm-Liouville systems was
developed. In particular it was shown that the set of all normalized
eigenfunctions forms an orthonormal basis for $L_{q}^{2}[0,a]$. Therefore
Lemma \ref{juan} can be used to show that the Fourier expansions in terms of
the eigenfunctions of $q$-Sturm-Liouville systems are pointwise convergent.
\end{remark}

\section{Uniform convergence}

\label{sec-uni-con}

By (\ref{fourier}) and (\ref{cfourier}) one may write, with $%
\,\eta_{k,\nu}\, $ given by (\ref{eta}),
\begin{equation}  \label{q-series}
\begin{array}{lll}
\displaystyle S_{q}^{(\nu)}[f](q^{n}) & = & \displaystyle\sum_{k=1}^{\infty}
a_{k}^{(\nu)}\left(f\right)J_{\nu }\big(q^{n+1}j_{k\nu };q^{2}\big) \\%
[1.2em]
& = & \displaystyle\sum_{k=1}^{\infty }\left(\frac{1}{\eta_{k,\nu}}%
\int_{0}^{1} tf(t)J_{\nu }\big(qj_{k\nu }t;q^{2}\big)d_{q}t\right) J_{\nu}%
\big(q^{n+1}j_{k\nu };q^{2}\big).%
\end{array}%
\end{equation}

\subsection{Behavior of $\,J_{\protect\nu}\big(q^{n+1}j_{k\protect\nu %
};q^{2} \big)$}

The study of the factor $\,J_{\nu}\big(q^{n+1}j_{k\nu };q^{2}\big)\,$ will
be crucial. We begin with the basic difference relation (2.12) of \cite[p.
693]{KS}, make the shift $\,q\to q^2\,$,
\begin{equation*}
J_{\nu}\big(q^2x;q^{2}\big)+q^{-\nu}\big(q^2x^2-1-q^{2\nu}\big)J_{\nu}\big(%
qx;q^{2}\big)+ J_{\nu}\big(x;q^{2}\big)=0\,,
\end{equation*}
and then use induction on $\,n\,$ to prove the following proposition.

\begin{proposition}
\label{prop-inicial}For $n=0,1,2,\cdots $,
\begin{equation*}
J_{\nu }\big(q^{n+1}j_{k\nu };q^{2}\big)=J_{\nu }\big(qj_{k\nu };q^{2}\big)%
P_{n}\big(j_{k\nu }^{2};q\big)\quad ,\qquad n=0,1,2,\cdots
\end{equation*}%
where $\,\left\{ P_{n}(x;q)\right\} _{n}\,$ is a sequence of polynomials
such that, for each $\,n=0,1,2,\cdots \,,$ $\,P_{n}(x;q)\,$ has degree $%
\,n\, $ in the variable $\,x\,$ and
\begin{equation*}
\left\{
\begin{array}{l}
P_{n+1}\big(j_{k\nu }^{2};q\big)=\Big\{\big(q^{\nu }+q^{-\nu }\big)-q^{-\nu
+2(n+1)}j_{k\nu }^{2}\Big\}P_{n}\big(j_{k\nu }^{2};q\big)-P_{n-1}\big(%
j_{k\nu }^{2};q\big)\,, \\[1em]
P_{0}(j_{k\nu }^{2};q)=1\,,\quad P_{-1}(j_{k\nu }^{2};q)=0\,.%
\end{array}%
\right.
\end{equation*}
\end{proposition}

Let
\begin{equation}
P_{n}\big(j_{k\nu }^{2};q\big):=\sum_{j=0}^{n}a_{j}^{(n,\nu )}(q)\big(%
j_{k\nu }^{2}\big)^{j}\text{.}  \label{notation}
\end{equation}%
We have the following recurrence relation for the polynomial coefficients $%
\,a_{j}^{(n,\nu )}\equiv a_{j}^{(n,\nu )}(q)$:
\begin{equation}
\left\{
\begin{array}{l}
a_{j}^{(n+1,\nu )}=\big(q^{\nu }+q^{-\nu }\big)a_{j}^{(n,\nu )}-q^{-\nu
+2(n+1)}a_{j-1}^{(n,\nu )}-a_{j}^{(n-1,\nu )}\;\;,\quad j\leq n, \\[1em]
a_{-1}^{(0,\nu )}=0\;\quad ,\quad a_{0}^{(0,\nu )}=1, \\[0.8em]
a_{j}^{(n,\nu )}=0\quad \mbox{whenever}\quad j>n\,,%
\end{array}%
\right.  \label{recurrence-coefficients}
\end{equation}%
Moreover, it follows from (\ref{recurrence-coefficients}) that, for every
integer $\,n\,$,
\begin{equation}
a_{0}^{(n,\nu )}=q^{-n\nu }\sum_{i=0}^{n}\big(q^{2\nu }\big)^{i}\,\quad %
\mbox{and}\quad a_{n}^{(n,\nu )}=(-1)^{n}q^{n(n+1-\nu )}\,.
\label{recurrence-coefficientsnn}
\end{equation}%
From (\ref{recurrence-coefficients}) one also deduces that

\begin{equation}  \label{recurrence-coefficients-nova}
\left\{
\begin{array}{l}
\displaystyle a_j^{(n,\nu)}=-q^{2-\nu}\sum_{\lambda=0}^{n-j}q^{2(n-1-%
\lambda)} a_0^{(\lambda,\nu)}a_{j-1}^{(n-1-\lambda,\nu)}\;\;,\quad j\leq n \\%
[1em]
a_0^{(0,\nu)}=1 \\[0.8em]
a_j^{(n,\nu)}=0 \quad \mbox{whenever} \quad j>n\,,%
\end{array}
\right.
\end{equation}
%\begin{equation}\label{recurrence-coefficients-nova}
%a_j^{(n,\nu)}=-q^{2-\nu}\sum_{\lambda=0}^{n-j}q^{2(n-1-\lambda)}a_0^{(\lambda,\nu)}a_{j-1}^{(n-1-\lambda,\nu)}\,.
%\end{equation}
%We intent to exhibit explicitly an expression for the coefficients $\,a_j^{(n,\nu)}\,$.

Now we are able to prove the following result for the $\,a_j^{(n,\nu)}$.

\vspace{0.7em}

\begin{proposition}
\label{fundamental-relation} An explicit expression for the polynomial
coefficients $\,a_j^{(n,\nu)}$ is given by
%\begin{equation}\label{newrelation}
%\!a_{j}^{(n,\nu)}\!=\!(\!-\!1)^{j}q^{j(j+1-\nu)}
%\!\sum_{i=0}^{\!\left[\frac{n-j}{2}\right]\!}\!a_{0}^{(n\!-\!j\!-\!2i,\nu)}q^{2i}
%\frac{\!\left(\!\left(q^2\right)^{\!j}\!;q^2\!\right)_{\!i}}{\left(q^2;q^2\right)_{i}}
%\frac{\!\left(\!\left(q^2\right)^{\!1\!+\!j}\!;q^2\!\right)_{\!n\!-\!j\!-\!2i}}{\left(q^2;q^2\right)_{n\!-\!j\!-\!2i}}
%\frac{\!\left(\!\left(q^2\right)^{\!1\!+\!n\!-\!2i}\!;q^2\!\right)_{\!i}}{
%\!\left(\!\left(q^2\right)^{n\!-\!j\!-\!2i\!+\!2}\!;q^2\!\right)_{\!i}}
%\end{equation}
\begin{equation*}
\begin{array}{l}
\label{newrelation} \displaystyle\!\!\!a_{j}^{(n,\nu)}\!=\!(\!-%
\!1)^{j}q^{j(j+1-\nu)} \!\sum_{i=0}^{\!\left[\frac{n-j}{2}\right]%
\!}\!a_{0}^{(n\!-\!j\!-\!2i,\nu)}q^{2i} \frac{\!\left(\!\left(q^2\right)^{%
\!j}\!;q^2\!\right)_{\!i}}{\left(q^2;q^2\right)_{i}} \frac{%
\!\left(\!\left(q^2\right)^{\!1\!+\!j}\!;q^2\!\right)_{\!n\!-\!j\!-\!2i}}{%
\left(q^2;q^2\right)_{n\!-\!j\!-\!2i}} \frac{\!\left(\!\left(q^2\right)^{\!1%
\!+\!n\!-\!2i}\!;q^2\!\right)_{\!i}}{ \!\left(\!\left(q^2\right)^{n\!-\!j\!-%
\!2i\!+\!2}\!;q^2\!\right)_{\!i}} \\[1em]
\displaystyle\hspace{2em}=\!(\!-\!1)^{j}q^{j(j+1-\nu)} \!\sum_{i=0}^{\!\left[%
\frac{n-j}{2}\right]\!}\!a_{0}^{(n\!-\!j\!-\!2i,\nu)}q^{2i} \frac{%
\!\left(\!\left(q^2\right)^{\!j}\!;q^2\!\right)_{\!i}}{\left(q^2;q^2%
\right)_{i}} \frac{\!\left(\!\left(q^2\right)^{\!1\!+\!j}\!;q^2\!\right)_{%
\!n\!-\!j\!-\!i}}{\left(q^2;q^2\right)_{n\!-\!j\!-\!i}} \frac{%
\!\left(\!\left(q^2\right)^{\!1\!+\!n\!-\!j\!-\!2i}\!;q^2\!\right)_{\!1}}{
\!\left(\!\big(q^2\big)^{\!1\!+\!n\!-\!j\!-\!i}\!;q^2\!\right)_{\!1}},%
\end{array}%
\end{equation*}
with $\,0\leq j\leq n\,$, $\,n=0,1,2,\cdots\,.$
\end{proposition}

\begin{proof}
The proof is carried out by induction and is rather long and technical, thus
we will simply present a sketch.

When $\,n=0\,$ the corresponding proposition is true. We point out that
Proposition \ref{newrelation} is clearly true for every $\,n\,$ when $%
\,j=0\, $. Let us, now, admit that
\begin{equation*}
\!a_{l}^{(k,\nu )}=(\!-\!1)^{l}q^{l(l+1-\nu )}\!\sum_{i=0}^{\!\left[ \frac{%
k-l}{2}\right] \!}\!a_{0}^{(k\!-\!l\!-\!2i,\nu )}q^{2i}\frac{\!\left(
\!\left( q^{2}\right) ^{\!l}\!;q^{2}\!\right) _{\!i}}{\left(
q^{2};q^{2}\right) _{i}}\frac{\!\left( \!\left( q^{2}\right)
^{\!1\!+\!l}\!;q^{2}\!\right) _{\!k\!-\!l\!-\!2i}}{\left( q^{2};q^{2}\right)
_{k\!-\!l\!-\!2i}}\frac{\!\left( \!\left( q^{2}\right)
^{\!1\!+\!k\!-\!2i}\!;q^{2}\!\right) _{\!i}}{\!\left( \!\left( q^{2}\right)
^{k\!-\!l\!-\!2i\!+2}\!;q^{2}\!\right) _{\!i}}
\end{equation*}%
holds true for $\,k=0,1,2,\cdots ,n-1:$ and $:0<l\leq k\,$. Then, using (\ref%
{recurrence-coefficients-nova}) and the hypothesis, one gets, for $\,0<j\leq
n\,$,
\begin{equation*}
\begin{array}{lll}
a_{j}^{(n,\nu )} & = & \displaystyle-q^{2-\nu }\sum_{\lambda
=0}^{n-j}q^{2(n-1-\lambda )}a_{0}^{(\lambda ,\nu )}(-1)^{j-1}q^{(j-1)(j-\nu
)}\sum_{i=0}^{\!\left[ \frac{n-j-\lambda }{2}\right] }a_{0}^{(n-j-\lambda
-2i,\nu )}c_{\lambda ,i} \\[1em]
& = & \displaystyle(-1)^{j}q^{j(j+1-\nu )}\sum_{\lambda =0}^{n-j}\left(
\sum_{i=0}^{\left[ \frac{n-j-\lambda }{2}\right] }\big(q^{2}\big)%
^{n-j-\lambda }a_{0}^{(\lambda ,\nu )}a_{0}^{(n-j-\lambda -2i,\nu
)}c_{\lambda ,i}\right)%
\end{array}%
\end{equation*}%
with $\;:c_{\lambda ,i}=q^{2i}\frac{\left( \left( q^{2}\right)
^{j-1};q^{2}\right) _{i}}{\left( q^{2};q^{2}\right) _{i}}\frac{\left( \left(
q^{2}\right) ^{j};q^{2}\right) _{n-j-\lambda -2i}}{\left( q^{2};q^{2}\right)
_{n-j-\lambda -2i}}\frac{\left( \left( q^{2}\right) ^{n-\lambda
-2i};q^{2}\right) _{i}}{\left( \left( q^{2}\right) ^{n-j-\lambda
-2i+2};q^{2}\right) _{i}}\,,\;$ hence
\begin{equation*}
a_{j}^{(n,\nu )}=(-1)^{j}q^{j(j+1-\nu )}\sum_{i=0}^{\left[ \frac{n-j}{2}%
\right] }\left( \sum_{\lambda =0}^{n-j-2i}\big(q^{2}\big)^{n-j-\lambda
}a_{0}^{(\lambda ,\nu )}a_{0}^{(n-j-\lambda -2i,\nu )}c_{\lambda ,i}\right) .
\end{equation*}%
Considering $\;\gamma _{\lambda ,i}=\big(q^{2}\big)^{n-j-\lambda
-i}c_{\lambda ,i}\;$ in the last identity and, then, using Lemma \ref%
{product-coefficients-00} it results
\begin{equation*}
a_{j}^{(n,\nu )}=(-1)^{j}q^{j(j+1-\nu )}\sum_{\lambda =0}^{\left[ \frac{n-j}{%
2}\right] }\big(q^{2}\big)^{i}\left( \sum_{\theta =0}^{\left[ \frac{n-j}{2}%
\right] -i}a_{0}^{(n-j-2i-2\theta ,\nu )}\left( \sum_{\lambda
=0}^{n-j-2i-\theta }\gamma _{\lambda ,i}\right) \right) \,,
\end{equation*}%
which can be transformed into
\begin{equation*}
\begin{array}{l}
\displaystyle a_{j}^{(n,\nu )}=(-1)^{j}q^{j(j+1-\nu )}\sum_{i=0}^{\left[
\frac{n-j}{2}\right] }a_{0}^{(n-j-2i,\nu )}\big(q^{2}\big)^{i}\frac{\left( %
\big(q^{2}\big)^{j};q^{2}\right) _{i}}{\left( q^{2};q^{2}\right) _{i}}\times
\\[1em]
\hspace{0.3em}\displaystyle\left( \sum_{\theta =0}^{i}\big(q^{2}\big)%
^{2\theta }\frac{\left( \big(q^{2}\big)^{j-1};q^{2}\right) _{\theta }}{%
\left( q^{2};q^{2}\right) _{\theta }}\left( \sum_{\lambda =0}^{n-j-2i}\!\big(%
q^{2}\big)^{\lambda }\frac{\left( \big(q^{2}\big)^{j};q^{2}\right)
_{i+\lambda }}{\left( q^{2};q^{2}\right) _{i+\lambda }}\frac{\left( \big(%
q^{2}\big)^{1+i+\lambda -\theta };q^{2}\right) _{1}}{\left( \big(q^{2}\big)%
^{1+i+\lambda };q^{2}\right) _{1}}\right) \!\right) ,%
\end{array}%
\end{equation*}%
and the proposition follows by Lemma \ref{finit-sum}.
\end{proof}

\begin{remark}
Notice that Proposition \ref{fundamental-relation} holds true for every
nonnegative integers $\,n\,$ and $\,j\,$ since, when $\,j>n\,$, then, by (%
\ref{recurrence-coefficients}), both members become identically zero.
\end{remark}

%\subsection{Behavior of $\,J_{\nu}\big(q^{1+n}j_{k\nu };q^{2}\big)$}

We also need the following result.

\begin{lemma}
\label{uniform-bounded-sequence} For $\,n=0,1,2,\ldots\,$, the sequence $\,%
\displaystyle\left\{J_{\nu}\big(q^{1+n}j_{k\nu };q^{2}\big)\right\}_{k\in%
\mathbb{N}}$ is uniformly bounded (with respect to $\,n$) whenever we fix $%
\,\nu>0$.
\end{lemma}

\begin{proof}
One must show that there exists $\,C\,$, independent of $\,k\,$ and $\,n\,$,
such that
\begin{equation*}
\left\vert J_{\nu }\big(q^{1+n}j_{k\nu };q^{2}\big)\right\vert \leq C
\end{equation*}%
for every $\,k\,$ and $\,n\,$. Using Theorem A we may write, for $k$ large
enough,
\begin{equation*}
J_{\nu }\big(q^{1+n}j_{k\nu };q^{2}\big)=J_{\nu }\big(q^{1+n-k+\epsilon
_{k}^{(\nu )}};q^{2}\big).
\end{equation*}%
Thus, by (\ref{zerosbound1}) and (\ref{boundasymptotic}), the last equality
puts in evidence that $\,n-k\,$ will play a crucial role on the behavior of $%
\,\displaystyle J_{\nu }\big(q^{1+n}j_{k\nu };q^{2}\big).$

\noindent We will separate the proof in two parts.

(i)\: When $k-n>0$ is sufficiently large then, by Corollary 3 of \cite{JLC4}%
,
\begin{equation*}
j_{k-n-1,\nu}=q^{1+n-k+\epsilon_{k-n-1}^{(\nu)}}<
q^{1+n-k+\epsilon_{k}^{(\nu)}}=q^{1+n}j_{k\nu}<q^{1+n-k}.
\end{equation*}
Therefore, being $\,n\in\mathbb{N}\,$ then
%\begin{equation}\label{qn-shifted-zeros}
$q^{1+n}j_{k\nu}\in \left]j_{k-n-1,\nu}\,,\,q^{1+n-k}\right[\,$,
%\end{equation}
whenever $k-n>0$ is sufficiently large.

\noindent On one hand, by the definition of $\,j_{k-n-1,\nu}\,$,
\begin{equation}  \label{R1}
J_{\nu}\big(j_{k-n-1,\nu};q^{2}\big)=0,
\end{equation}
and, on the other hand, by (12) of \cite[p. 1205]{BBEB}, for large
(positive) values of $\,k-n\,$,
\begin{equation}  \label{R2}
J_{\nu}\big(q^{1+n-k};q^2\big)\leq \frac{\left(-q^2,-q^{2(\nu+1)};q^2%
\right)_{\infty}}{\left(q^2;q^2\right)_{\infty}} q^{(k-n+\nu)(k-n-1)}.
\end{equation}
However, again by Theorem A,
\begin{equation*}
\left]j_{k-n-1,\nu}\,,\,q^{1+n-k}\right[\subset\left]q^{1+n-k+%
\alpha_{k-n-1}^{(\nu)}},q^{1+n-k}\right[.
\end{equation*}
Thus, by Corollary 2 of \cite{JLC4}, the function $\,J_{\nu}\big(x;q^{2}\big)%
\,$ is monotone in the interval $\,\left]j_{k-n-1,\nu}\,,\,q^{1+n-k}\right[%
\, $, hence, by (\ref{R1}) and (\ref{R2}), $\,\left|J_{\nu}\big(%
q^{1+n}j_{k\nu };q^{2}\big)\right|\,$ is bounded whenever $\,k-n\,$ is
sufficiently large (positive).

\vspace{0.3em} (ii)\: Now, let us consider all the other possible cases for $%
\,k-n\,$, i.e., the cases where $\,k-n\,$ is bounded above. Then $\,n-k\,$
is bounded below, thus, if $\,\nu>0\,$, $\,\left|J_{\nu}\big(q^{1+n}j_{k\nu
};q^{2}\big)\right|= \left|J_{\nu}\big(q^{1+n-k+\epsilon_{k}^{(\nu)}};q^{2}%
\big)\right|\,$ is trivially bounded for all such cases.

Joining both cases (i) and (ii) and since all the possible cases for $\,k\,$
and $\,n\,$ were considered, the lemma follows.
\end{proof}

\subsection{Behavior of $\,J_{\protect\nu}^{\prime}\big(qj_{k\protect\nu %
};q^{2}\big)$}

\noindent The asymptotic behavior of $\,J_{\nu }^{\prime }\big(qj_{k\nu
};q^{2}\big)\,$ when $\,k\rightarrow \infty \,$ was recently obtained \cite[%
Lemma 1]{JLC4}. The proof combines the asymptotic properties of the infinite$%
\,q$-shifted factorial $\,(z;q)\,$ (or the $\,q$-Pochhammer symbol) from
\cite{Daalhuis} with the ideas developed in \cite{SS2016}.

\noindent Nevertheless, we present the next lemma which is equivalent to the
afore mentioned result and provides a different direct proof based only on
the definition of the Hahn-Exton $\,q$-Bessel function and its derivative.

\begin{lemma}
\label{derivative} For large values of $\,k\,$,
\begin{equation*}
J_{\nu }^{\prime }(j_{k\nu };q^{2})=A_{\nu }(q)\,q^{-\big(k+\frac{\nu }{2}%
-1-\epsilon _{k}^{(\nu )}\big)^{2}}S_{k}\,,
\end{equation*}

where $\;A_{\nu}(q)\!=\!\frac{2\big(q^{2(\nu+1)};q^2\big)_{\infty}}{\big(%
q^2;q^2\big)_{\infty}} \,q^{\frac{(\nu-1)(\nu-3)}{4}}\;$ and
%, at least for $% \:0<q\leq\!\big(\frac{1}{51}\big)^{\frac{1}{50}}\:,$
$\;\displaystyle\liminf_{k\rightarrow\infty}|S_k|\!>\!0\,.$
\end{lemma}

\begin{proof}
We will present only the main steps of the proof.
%Its technique is similar to the corresponding one of the Theorem 6.1 \cite[page 147]{BC}.
Computing the derivative of the function $\,J_{\nu }(z;q^{2})\,$ and
considering $\,z=j_{k\nu }$,
\begin{equation*}
J_{\nu }^{\prime }(j_{k\nu };q^{2})=\frac{2\big(q^{2(\nu +1)};q^{2}\big)%
_{\infty }}{\big(q^{2};q^{2}\big)_{\infty }}\sum\limits_{n=0}^{\infty
}(-1)^{n}\frac{nq^{n(n+1)}}{(q^{2(\nu +1)},q^{2};q^{2})_{n}}\big(j_{k\nu }%
\big)^{2n+\nu -1}\,.
\end{equation*}%
By Theorem A we may write $\,j_{k\nu }=q^{-k+\epsilon _{k}^{(\nu )}}\,$ so,
the above identity becomes
\begin{equation}
J_{\nu }^{\prime }(j_{k\nu };q^{2})=A_{\nu }(q)\,q^{-\big(k+\frac{\nu }{2}%
-1-\epsilon _{k}^{(\nu )}\big)^{2}}S_{k}\,,  \label{comportamento}
\end{equation}%
where $\;A_{\nu }(q)\!=\!\frac{2\big(q^{2(\nu +1)};q^{2}\big)_{\infty }}{%
\big(q^{2};q^{2}\big)_{\infty }}\,q^{\frac{(\nu -1)(\nu -3)}{4}}\;$ and
%\begin{equation*}
$\displaystyle S_{k}=\sum\limits_{n=0}^{\infty }(-1)^{n}\frac{nq^{\big(%
n-k+1/2+\epsilon _{k}^{(\nu )}\big)^{2}}}{(q^{2(\nu +1)},q^{2};q^{2})_{n}}$%
\thinspace . %\label{Sk}
%\end{equation*}%
Considering $\,m=n-k$, straightforward manipulations give
\begin{equation*}
(-1)^{k}S_{k}=\sum\limits_{m=-k}^{\infty }(-1)^{m}\frac{mq^{\big(%
m+1/2+\epsilon _{k}^{(\nu )}\big)^{2}}}{(q^{2(\nu +1)},q^{2};q^{2})_{m+k}}\,.
\end{equation*}%
Thus
\begin{equation*}
(-1)^{k}S_{k}=\sum\limits_{m=-k}^{-(2p+2)}F_{m,k}^{(\nu
)}(q)-\sum\limits_{m=-(2p+1)}^{2p}F_{m,k}^{(\nu
)}(q)+\sum\limits_{m=2p+1}^{\infty }F_{m,k}^{(\nu )}(q)
\end{equation*}%
with
\begin{equation*}
F_{m,k}^{(\nu )}(q)=(-1)^{m}\frac{mq^{\big(m+1/2+\epsilon _{k}^{(\nu )}\big)%
^{2}}}{(q^{2(\nu +1)},q^{2};q^{2})_{m+k}}\,.
\end{equation*}%
Hence,
\begin{equation*}
\left\vert S_{k}\right\vert \geq \left\vert
\sum\limits_{m=-(2p+1)}^{2p}F_{m,k}^{(\nu )}(q)\right\vert -\left\vert
\sum\limits_{m=-k}^{-(2p+2)}F_{m,k}^{(\nu )}(q)\right\vert +\left\vert
\sum\limits_{m=2p+1}^{\infty }F_{m,k}^{(\nu )}(q)\right\vert \,.
\end{equation*}%
Since$\,p\,$ is an arbitrary positive integer and since by (\ref{zerosbound1}%
)-(\ref{boundasymptotic}), $\,\epsilon _{k}^{(\nu )}=\epsilon _{k}^{(\nu
)}(q^{2})\rightarrow 0$ when $\,k\rightarrow \infty $, we have
\begin{equation*}
\left\vert S_{k}\right\vert \geq \frac{q^{1/4}}{\big(q^{2(\nu
+1)},q^{2};q^{2}\big)_{\infty }}\,\left\vert \sum\limits_{i=0}^{\infty
}(-1)^{i}(2i+1)q^{i(i+1)}\right\vert \,.
\end{equation*}%
Identity (10.4.9) of Corollary 10.4.2 due to Jacobi \cite[page 500]{AAR}
guarantees that
\begin{equation}
\sum\limits_{i=0}^{\infty }(-1)^{i}(2i+1)q^{i(i+1)}=\prod_{i=1}^{\infty }%
\big(1-q^{2i}\big)^{3}>0  \label{positivity}
\end{equation}%
thus, by (\ref{comportamento})-(\ref{positivity}), the Lemma follows.
\end{proof}

Notice that from the above Lemma it follows that $J_{\nu }^{\prime }(j_{k\nu
};q^{2})=\mathcal{O}\left( q^{-k(k+\nu-2)}\right)$ as $k\to\infty$.

\subsection{Sufficient conditions}

With the previous notation $\,V_{q}^{+}\,=\,\left\{q^{n}:\,n=0,1,2,\ldots
\right\}$, which coincides with the support points of the $\,q$-integral (%
\ref{qintegral}) in $\,[0,1]\,,$ we will consider the following $\,q$-linear
H\"{o}lder concept \cite[p. 103]{JLC2} adapted to the set of points $%
\,V_{q}^{+}\,$.

\vspace{0.7em} \noindent$\mathbf{Definition.}$ \textsl{If two constants $%
\,M\,$ and $\,\lambda\,$ exist such that
\begin{equation*}
\Big|f\big(q^{n-1}\big)-f\big(q^{n}\big)\Big|\leq M q^{\lambda n}\:,\quad
n=0,1,2,\ldots\,,
\end{equation*}
then the function $\,f\,$ is said to be $\,q$-linear H\"older of order $%
\,\lambda\,$ (in $\,V_q^+\cup \{q^{-1}\}$).}

\vspace{0.4em} In \cite{JLC3} the following upper bound for basic
Fourier-Bessel coefficient (\ref{cfourier}) has been obtained.

\textbf{Theorem.} \emph{If the function} $\,f\,$ \emph{is} $\,q$-\emph{%
linear H\"{o}lder of order} $\,\alpha >0\,$ \emph{in} $\,V_{q}^{+}\,$ \emph{%
and such that} $\,t^{-\frac{1}{2}}f(t)\in L_{q}^{2}(0,1)\,$ \emph{and the
limit} $\,\displaystyle\lim_{x\rightarrow 0^{+}}f(x)=f(0^{+})$ \emph{is
finite then}
\begin{equation*}
\begin{array}{l}
\displaystyle\left\vert \int_{0}^{1}tf(t)J_{\nu }(qj_{k\nu
}t;q^{2})d_{q}t\right\vert \leq \frac{(1-q)q^{\nu -1}}{j_{k\nu }}\left\vert f%
\big(q^{-1}\big)J_{\nu +1}(qj_{k\nu };q^{2})\right\vert + \\[1em]
\hspace{9em}\displaystyle\frac{(1-q)q^{\frac{\nu -3}{2}}}{j_{k\nu }}\eta
_{k,\nu }^{\frac{1}{2}}\left( \frac{q^{\frac{\nu +1}{2}}M_{1}}{(1-q)^{\frac{1%
}{2}}\big(1-q^{2\alpha }\big)^{\frac{1}{2}}}+\frac{q^{\frac{\nu }{2}}-q^{-%
\frac{\nu }{2}}}{q^{\frac{1}{2}}-q^{-\frac{1}{2}}}\sqrt{M_{2}}\right) \,,%
\end{array}%
\end{equation*}%
\emph{where} $\,M_{1}\,$ \emph{and} $\,M_{2}\,$ \emph{are independent of} $%
\,k$ and $\,\eta _{k,\nu }\,$ is given by (\ref{eta}).

\noindent However, the conditions on the function $\,f\,$ stated in this
theorem seem to be not sufficient to obtain the uniform convergence of its
basic Fourier-Bessel expansion and we will need to impose the slightly more
restrictive conditions of the Theorem \ref{uniform-convergence}. First we
need the following lemma.

\begin{lemma}
\label{q-integral-of-coef-fourier-new} If $\,f\,$ is a function such that $\,%
\displaystyle\lim_{x\to 0^+}f(x)<+\infty\,$ and $\,\nu>0\,$ then
\begin{equation*}
\begin{array}{l}
\displaystyle\int_{0}^{1}tf(t)J_{\nu }(qj_{k\nu}t;q^{2})d_{q}t= \frac{%
(1-q)q^{\nu-2}f\big(q^{-1}\big)J_{\nu}(qj_{k\nu};q^{2})}{j_{k\nu}^2} -\frac{%
(1-q)^2q^{\nu-3}}{\left(q^{\frac{1}{2}}-q^{-\frac{1}{2}}\right)^2j_{k\nu}^2}%
\times \\[1em]
\displaystyle\left[ \left(q^{\frac{\nu}{2}}-q^{-\frac{\nu}{2}}\right)
\left(q^{\frac{\nu}{2}}\!\int_{0}^{1}\!J_{\nu}(qj_{k\nu}t;q^2)\frac{f(qt)}{t}%
d_qt- q^{-\frac{\nu}{2}}\!\int_{0}^{1}\!J_{\nu}(qj_{k\nu}t;q^2)\frac{f(t)}{t}%
d_qt\right) \right.- \\[1em]
\displaystyle\left.q^{\frac{\nu}{2}} \left(q^{\frac{\nu}{2}%
}\!\int_{0}^{1}\!\!J_{\nu}(qj_{k\nu}t;q^2)\frac{f(qt)-f(t)}{t}d_qt- q^{-%
\frac{\nu}{2}}\!\int_{0}^{1}\!\!J_{\nu}(qj_{k\nu}t;q^2)\frac{f(t)-f(t/q)}{t}%
d_qt\right) \right],%
\end{array}%
\end{equation*}
whenever the involved $\,q$-integrals exist.
\end{lemma}

\begin{proof}
Using first (3.7) of \cite[Proposition 4, p. 7]{JLC3}
\begin{equation*}
\frac{\delta_q\left[x^{\nu}J_{\nu}\big(x;q^2\big)\right]}{\delta_q x}= \frac{%
q^{-\frac{\nu}{2}}}{1-q}x^{\nu}J_{\nu-1}\big(q^{-\frac{1}{2}}x;q^2\big),
\end{equation*}
then, the $\,q$-integration by parts formula \eqref{q-integration-by-parts}
and the hypothesis $\,\displaystyle\lim_{x\to 0^+}f(x)=f(0^+)\,$ one gets,
\begin{equation}  \label{starting-integral}
\begin{array}{l}
\displaystyle\int_{0}^{1}tf(t)J_{\nu }(qj_{k\nu}t;q^{2})d_{q}t= \frac{%
(1-q)q^{\nu-1}f\big(q^{-1}\big)J_{\nu+1}(qj_{k\nu};q^{2})}{j_{k\nu}} -\frac{%
(1-q)q^{\frac{\nu-3}{2}}}{j_{k\nu}}\times \\[1em]
\displaystyle\left[ -\frac{q^{\frac{\nu}{2}}-q^{-\frac{\nu}{2}}}{q^{\frac{1}{%
2}}-q^{-\frac{1}{2}}} \int_{0}^{1}\!J_{\nu+1}(qj_{k\nu}t;q^2)f(t)d_qt+ q^{%
\frac{\nu}{2}}\!\int_{0}^{1}\!J_{\nu+1}(qj_{k\nu}t;q^2)f_2(t)d_qt\right].%
\end{array}%
\end{equation}
where
\begin{equation}  \label{def-f2}
f_2(t):=\frac{t\delta_qf\big(q^{-\frac 1 2}t\big)}{\delta_qt},
\end{equation}
and the operator $\delta_q$ is given in %
\eqref{symmetric-q-difference-operator}. For more details in the previous
calculations see the proof of Theorem 1 \cite[p. 10]{JLC3}.

\; (i) \emph{Step 1}: for the first $\,q$-integral that figures in the right
side of (\ref{starting-integral}): using (3.8) of \cite[Proposition 4, p. 7]%
{JLC3} and again \cite[Lemma 2, p. 5]{JLC3} together with both hypothesis $\,%
\displaystyle\lim_{x\to 0^+}f(x)=f(0^+)\,$ and $\,\nu>0\,$, it results
\begin{equation*}
\int_{0}^{1}\!J_{\nu+1}(qj_{k\nu}t;q^2)f(t)d_qt=\frac{(1-q)q^{\frac{\nu-3}{2}%
}}{j_{k\nu}} \int_{0}^{1}\!J_{\nu}(qj_{k\nu}t;q^2) \frac{t^{-\nu}\delta_q%
\left[t^{\nu}f\big(q^{\frac 1 2}t\big)\right]}{\delta_qt}d_qt\,,
\end{equation*}
which, by (\ref{symmetric-q-derivative-operator}), becomes
\begin{equation}  \label{integral-1}
\begin{array}{l}
\displaystyle\int_{0}^{1}\!J_{\nu+1}(qj_{k\nu}t;q^2)f(t)d_qt= \frac{(1-q)q^{%
\frac{\nu-3}{2}}}{\left(q^{\frac{1}{2}}-q^{-\frac{1}{2}}\right)j_{k\nu}}%
\times \\[1em]
\displaystyle\left[q^{\frac{\nu}{2}}\int_{0}^{1}\!J_{\nu}(qj_{k\nu}t;q^2)%
\frac{f(qt)}{t}d_qt- q^{-\frac{\nu}{2}} \int_{0}^{1}\!J_{\nu}(qj_{k\nu}t;q^2)%
\frac{f(t)}{t}d_qt\right].%
\end{array}%
\end{equation}

\; (i) \emph{Step 2}: in a similar way, for the second $\,q$-integral that
figures in the right side of (\ref{starting-integral}): using the same
hypothesis and a similar procedure used in \emph{Step 1}, together with (\ref%
{def-f2}), we obtain
\begin{equation}  \label{integral-2}
\begin{array}{l}
\displaystyle\int_{0}^{1}\!J_{\nu+1}(qj_{k\nu}t;q^2)f_2(t)d_qt= \frac{%
(1-q)q^{\frac{\nu-3}{2}}}{\left(q^{\frac{1}{2}}-q^{-\frac{1}{2}%
}\right)^2j_{k\nu}}\times \\[1em]
\displaystyle\left[q^{\frac{\nu}{2}}\int_{0}^{1}\!J_{\nu}(qj_{k\nu}t;q^2)%
\frac{f(qt)-f(t)}{t}d_qt- q^{-\frac{\nu}{2}}\int_{0}^{1}\!J_{\nu}(qj_{k%
\nu}t;q^2)\frac{f(t)-f\big(\frac t q\big)}{t}d_qt\right].%
\end{array}%
\end{equation}
The lemma follows by introducing (\ref{integral-1}) and (\ref{integral-2})
into (\ref{starting-integral}) and using identity (\ref{relation-vii}).
\end{proof}

\begin{remark}
Sufficient conditions to guarantee the existence of all the $\,q$-integrals
involved in the previous lemma coincide with the sufficient conditions of
Theorem \ref{uniform-convergence} for the uniform convergence of the $\,q$%
-Fourier-Bessel series.
\end{remark}

We are now in conditions to prove our main result, where sufficient
conditions on the function $\,f\,$ are given in order that its correspondent
$\,q$-Fourier-Bessel series converges uniformly to the function itself in
the set $\,V_{q}^{+}\,=\,\left\{ q^{n}:\,n=0,1,2,\ldots \right\} $.

\textbf{Proof of Theorem \ref{uniform-convergence}}
%\textbf{\label{uniform-convergence copy(1)}}.
Under the assumptions on $\,f\,$ we can use Lemma \ref%
{q-integral-of-coef-fourier-new} to obtain
\begin{equation}
\begin{array}{l}
\displaystyle\left\vert \int_{0}^{1}tf(t)J_{\nu }(qj_{k\nu
}t;q^{2})d_{q}t\right\vert \leq \frac{(1-q)q^{\nu -2}\left\vert f\big(\frac{1%
}{q}\big)J_{\nu }(qj_{k\nu };q^{2})\right\vert }{j_{k\nu }^{2}}+\frac{%
(1-q)^{2}q^{\nu -3}}{\left( q^{\frac{1}{2}}-q^{-\frac{1}{2}}\right)
^{2}j_{k\nu }^{2}}\times \\[1em]
\displaystyle\left[ \left\vert q^{\frac{\nu }{2}}-q^{-\frac{\nu }{2}%
}\right\vert \left( q^{\frac{\nu }{2}}\!\left\vert \int_{0}^{1}\!J_{\nu
}(qj_{k\nu }t;q^{2})\frac{f(qt)}{t}d_{q}t\right\vert +q^{-\frac{\nu }{2}%
}\!\left\vert \int_{0}^{1}\!J_{\nu }(qj_{k\nu }t;q^{2})\frac{f(t)}{t}%
d_{q}t\right\vert \right) \right. + \\[1em]
\displaystyle\left. q^{\frac{\nu }{2}}\!\left( \!q^{\frac{\nu }{2}%
}\!\left\vert \int_{0}^{1}\!\!J_{\nu }(qj_{k\nu }t;q^{2})\frac{f(qt)\!-\!f(t)%
}{t}d_{q}t\right\vert +q^{-\frac{\nu }{2}}\!\left\vert
\int_{0}^{1}\!\!J_{\nu }(qj_{k\nu }t;q^{2})\frac{f(t)\!-\!f\big(\frac{t}{q}%
\big)}{t}d_{q}t\right\vert \right) \right] \text{,}%
\end{array}
\label{superior-bound-qintegral}
\end{equation}%
Applying the $\,q$-H\"{o}lder type inequality of \cite[Th. 3.4, p. 346]{CP}
with $\,p=2\,$, we may write, for each of the $\,q$-integrals that figure in
the right side of the previous inequality, respectively,
\begin{equation}
\begin{array}{lll}
\displaystyle\left\vert \int_{0}^{1}J_{\nu }(qj_{k\nu }t;q^{2})\frac{f(qt)}{t%
}d_{q}t\right\vert & \leq & \displaystyle\left( \int_{0}^{1}tJ_{\nu
}^{2}(qj_{k\nu }t;q^{2})d_{q}t\right) ^{\frac{1}{2}}\displaystyle\left(
\int_{0}^{1}\frac{f^{2}(qt)}{t^{3}}d_{q}t\right) ^{\frac{1}{2}} \\[1em]
& \leq & \displaystyle\eta _{k\nu }^{\frac{1}{2}}\left( \int_{0}^{1}\frac{%
f^{2}(qt)}{t^{3}}d_{q}t\right) ^{\frac{1}{2}}\text{,}%
\end{array}
\label{CS-1}
\end{equation}%
\begin{equation}
\left\vert \int_{0}^{1}J_{\nu }(qj_{k\nu }t;q^{2})\frac{f(t)}{t}%
d_{q}t\right\vert \leq \displaystyle\eta _{k\nu }^{\frac{1}{2}}\left(
\int_{0}^{1}\frac{f^{2}(t)}{t^{3}}d_{q}t\right) ^{\frac{1}{2}},  \label{CS-2}
\end{equation}%
\begin{equation}
\left\vert \int_{0}^{1}J_{\nu }(qj_{k\nu }t;q^{2})\frac{f(qt)-f(t)}{t}%
d_{q}t\right\vert \leq \displaystyle\eta _{k\nu }^{\frac{1}{2}}\left(
\int_{0}^{1}\frac{\left( f(qt)-f(t)\right) ^{2}}{t^{3}}d_{q}t\right) ^{\frac{%
1}{2}},  \label{CS-3}
\end{equation}%
\begin{equation}
\left\vert \int_{0}^{1}J_{\nu }(qj_{k\nu }t;q^{2})\frac{f(t)-f\big(\frac{t}{q%
}\big)}{t}d_{q}t\right\vert \leq \displaystyle\eta _{k\nu }^{\frac{1}{2}%
}\left( \int_{0}^{1}\frac{\left( f(t)-f\big(\frac{t}{q}\big)\right) ^{2}}{%
t^{3}}d_{q}t\right) ^{\frac{1}{2}}.  \label{CS-4}
\end{equation}%
Now, using the definition (\ref{qintegral}), since $t^{-\frac{3}{2}}f(t)\in
L_{q}^{2}[0,1]$,
\begin{equation}
\int_{0}^{1}\frac{f^{2}(qt)}{t^{3}}d_{q}t=(1-q)\sum_{n=0}^{\infty }\frac{%
q^{n}f^{2}(q^{n+1})}{q^{3n}}=(1-q)q^{2}\sum_{n=0}^{\infty }\left( \frac{%
f(q^{n+1})}{q^{n+1}}\right) ^{2}=S<+\infty \text{,}  \label{L2-1}
\end{equation}%
and
\begin{equation}
\int_{0}^{1}\frac{f^{2}(t)}{t^{3}}d_{q}t=(1-q)\sum_{n=0}^{\infty }\left(
\frac{f(q^{n})}{q^{n}}\right) ^{2}=T<+\infty \,\text{.}  \label{L2-2}
\end{equation}%
Moreover, since $\,f\,$ is $\,q$-linear H\"{o}lder of order $\,\alpha >1\,$
in $\,V_{q}^{+}\cup \left\{ q^{-1}\right\} \,$, there exist $\,M\,$ and $%
\,N\,$ such that
\begin{equation}
\begin{array}{lll}
\displaystyle\int_{0}^{1}\frac{\left( f(qt)-f(t)\right) ^{2}}{t^{3}}d_{q}t &
= & \displaystyle(1-q)\sum_{n=0}^{\infty }\frac{\Big(f\big(q^{n+1}\big)-f%
\big(q^{n}\big)\Big)^{2}}{q^{2n}} \\[1em]
& \leq & \displaystyle(1-q)M^{2}\sum_{n=0}^{\infty }\frac{q^{2\alpha n}}{%
q^{2n}}=\frac{(1-q)M^{2}}{1-q^{2(\alpha -1)}},%
\end{array}
\label{H-1}
\end{equation}%
and
\begin{equation}
\begin{array}{lll}
\displaystyle\int_{0}^{1}\frac{\left( f(t)-f\big(\frac{t}{q}\big)\right) ^{2}%
}{t^{3}}d_{q}t & = & \displaystyle(1-q)\sum_{n=0}^{\infty }\frac{\Big(f\big(%
q^{n}\big)-f\big(q^{n-1}\big)\Big)^{2}}{q^{2n}} \\[1em]
& \leq & \displaystyle(1-q)N^{2}\sum_{n=0}^{\infty }\frac{q^{2\alpha n}}{%
q^{2n}}=\displaystyle\frac{(1-q)N^{2}}{1-q^{2(\alpha -1)}}.%
\end{array}
\label{H-2}
\end{equation}%
The constants $\,S\equiv S_{q}(f)\,$, $\,T\equiv T_{q}(f)\,$ $\,M\equiv
M_{q}(f)\,$ and $\,N\equiv N_{q}(f)\,$ are independent of $\,k$. Introducing
inequalities (\ref{L2-1}), (\ref{L2-2}), (\ref{H-1}) and (\ref{H-2}) into
inequalities (\ref{CS-1}), (\ref{CS-2}), (\ref{CS-3}) and (\ref{CS-4}),
respectively, and the resulting ones into (\ref{superior-bound-qintegral}),
gives:
\begin{equation*}
\begin{array}{l}
\displaystyle\left\vert \int_{0}^{1}tf(t)J_{\nu }(qj_{k\nu
}t;q^{2})d_{q}t\right\vert \leq \frac{(1-q)q^{\nu -2}\left\vert f\big(\frac{1%
}{q}\big)J_{\nu }(qj_{k\nu };q^{2})\right\vert }{j_{k\nu }^{2}}+\frac{%
(1-q)^{2}q^{\nu -3}\eta _{k}^{\frac{1}{2}}}{\left( q^{\frac{1}{2}}-q^{-\frac{%
1}{2}}\right) ^{2}j_{k\nu }^{2}}\times \\[1em]
\displaystyle\left[ \left\vert q^{\frac{\nu }{2}}-q^{-\frac{\nu }{2}%
}\right\vert \left( q^{\frac{\nu }{2}}\sqrt{S}+q^{-\frac{\nu }{2}}\sqrt{T}%
\right) +\displaystyle q^{\frac{\nu }{2}}\left( q^{\frac{\nu }{2}}\sqrt{%
\frac{(1-q)M^{2}}{1-q^{2(\alpha -1)}}}+q^{-\frac{\nu }{2}}\sqrt{\frac{%
(1-q)N^{2}}{1-q^{2(\alpha -1)}}}\right) \!\right] ,%
\end{array}%
\end{equation*}%
or, equivalently,
\begin{equation*}
\left\vert \int_{0}^{1}tf(t)J_{\nu }\big(qj_{k\nu }t;q^{2}\big)%
d_{q}t\right\vert \leq \frac{1}{j_{k\nu }^{2}}\left\{ C_{1}\left\vert J_{\nu
}\big(qj_{k\nu };q^{2}\big)\right\vert +C_{2}\left( \eta _{k,\nu }\right) ^{%
\frac{1}{2}}\right\} ,
\end{equation*}%
where $\,C_{1}\,$ and $\,C_{2}\,$ depend on $\,f\,$, $\,\nu \,$ and $\,q\,$
but are independent of $\,k$. Thus, the absolute value of the $\,k^{th}\,$
term of the infinite sum (\ref{q-series}) verifies the following inequality
\begin{equation*}
\begin{array}{lll}
\displaystyle\left\vert a_{k}\left( f\right) J_{\nu }\big(q^{n+1}j_{k\nu
};q^{2}\big)\right\vert & = & \displaystyle\ \left\vert \left( \frac{1}{\eta
_{k,\nu }}\int_{0}^{1}tf(t)J_{\nu }\big(qj_{k\nu }t;q^{2}\big)d_{q}t\right)
J_{\nu }\big(q^{n+1}j_{k\nu };q^{2}\big)\right\vert \\[1em]
& \leq & \displaystyle\ \left\{ \frac{C_{1}\left\vert J_{\nu }\big(qj_{k\nu
};q^{2}\big)\right\vert }{j_{k\nu }^{2}\eta _{k,\nu }}+\frac{C_{2}}{j_{k\nu
}^{2}\eta _{k,\nu }^{\frac{1}{2}}}\right\} \left\vert J_{\nu }\big(%
q^{1+n}j_{k\nu };q^{2}\big)\right\vert .%
\end{array}%
\end{equation*}%
Therefore, by (\ref{eta}), $\left\vert a_{k}\left( f\right) J_{\nu }\big(%
q^{n+1}j_{k\nu };q^{2}\big)\right\vert \;$ is bounded by
\begin{equation}
\left\{ \frac{qC_{1}}{j_{k\nu }\left\vert J_{\nu }^{\prime }\big(j_{k\nu
};q^{2}\big)\right\vert }\!+\!\frac{q^{\frac{1}{2}}C_{2}}{j_{k\nu }^{\frac{3%
}{2}}\left\vert J_{\nu }^{\prime }\big(j_{k\nu };q^{2}\big)\right\vert ^{%
\frac{1}{2}}\left\vert J_{\nu }\big(qj_{k\nu };q^{2}\big)\right\vert ^{\frac{%
1}{2}}}\right\} \!\left\vert J_{\nu }\big(q^{1+n}j_{k\nu };q^{2}\big)%
\right\vert .  \label{con-uni-two}
\end{equation}%
For the first term, by using Theorem A and Lemmas \ref%
{uniform-bounded-sequence} and \ref{derivative} it follows that there exists
a constant $M>0$, independent of $k$, such that
\begin{equation*}
\frac{\left\vert J_{\nu }\big(q^{1+n}j_{k\nu };q^{2}\big)\right\vert }{%
j_{k\nu }\left\vert J_{\nu }^{\prime }\big(j_{k\nu };q^{2}\big)\right\vert }%
\leq Mq^{k}.
\end{equation*}%
Now the second term. By Proposition \ref{prop-inicial},
\begin{equation*}
\frac{\left\vert J_{\nu }\big(q^{1+n}j_{k\nu };q^{2}\big)\right\vert }{%
j_{k\nu }^{\frac{3}{2}}\left\vert J_{\nu }^{\prime }\big(j_{k\nu };q^{2}\big)%
\right\vert ^{\frac{1}{2}}\left\vert J_{\nu }\big(qj_{k\nu };q^{2}\big)%
\right\vert ^{\frac{1}{2}}}=\frac{\left\vert J_{\nu }\big(qj_{k\nu };q^{2}%
\big)\right\vert ^{\frac{1}{2}}}{j_{k\nu }^{\frac{3}{2}}\left\vert J_{\nu
}^{\prime }\big(j_{k\nu };q^{2}\big)\right\vert ^{\frac{1}{2}}}\left\vert
P_{n}\big(j_{k\nu }^{2};q\big)\right\vert .
\end{equation*}%
Using the expression \eqref{notation} as well as Proposition \ref%
{newrelation} and Eq. \eqref{recurrence-coefficientsnn} it follows that, for
all given $n$,
\begin{equation*}
\lim_{k\rightarrow \infty }\left\vert \frac{P_{n}\big(j_{k\nu }^{2};q\big)}{%
a_{n}^{(n,\nu )}(j_{k\nu }^{2})^{n}}\right\vert =1,\quad \mbox{and so,}\quad
P_{n}\big(j_{k\nu }^{2};q\big)=\mathcal{O}\left( q^{n(n+1-\nu )}j_{k\nu
}^{2n}\right) \quad \mbox{as}\quad k\rightarrow \infty .
\end{equation*}%
Substituting the above expression and combining Theorems A and C with Lemmas %
\ref{uniform-bounded-sequence} and \ref{derivative}, leads to the following
bound for the second term in \eqref{con-uni-two}:
\begin{equation*}
Aq^{k^{2}+\nu k+n(n+1-\nu )-2kn}=Aq^{\left( n-k-\frac{\nu -1}{2}\right) ^{2}-%
\frac{(\nu -1)^{2}}{4}}q^{k}\leq Bq^{k},
\end{equation*}%
where we have used $\epsilon _{k}^{(\nu )}(q)>0$ and $\epsilon _{k}^{(\nu
)}(q)=\mathcal{O}(q^{2k})$ when $k\rightarrow \infty $ (see %
\eqref{boundasymptotic} and \eqref{zerosbound1}). The constants $A$ and $B$
are positive and independent of $k$. Therefore,
\begin{equation*}
\left\vert a_{k}\left( f\right) J_{\nu }\big(q^{n+1}j_{k\nu };q^{2}\big)%
\right\vert =\mathcal{O}\left( q^{{k}}\right) \quad \mbox{as}\quad
k\rightarrow \infty ,
\end{equation*}%
which proves the uniform convergence of the basic Fourier-Bessel series (\ref%
{fourier}) on the set $\,V_{q}^{+}\,$.

\begin{remark}
We should point out that the $\,q$-linear H\"{o}lder condition does not need
to be strictly accomplish for all $\,n=0,1,2,\cdots \,.$ It suffices that $%
\,f\,$ satisfies the \emph{almost} $\,q$-linear H\"{o}lder condition \cite[%
p. 105]{JLC2}, which means that the condition only needs to be satisfied for
all integers $\,n\,$ such that $\,n\geq n_{0}\,$, where $\,n_{0}\,$ is a
positive integer.
\end{remark}

The next result is a corollary of Theorem B and Theorem \ref%
{uniform-convergence}.

\begin{theorem}
\label{FinalTheorem} If $\;f\in L_q^2[0,1]\,$ and the $\,q$-Fourier-Bessel
series $\,S_q^{(\nu)}[f](x)\,$ converges uniformly on $\,\displaystyle%
\,V_q^+=\left\{q^n:\:n=0,1,2,\ldots\,\right\}$, then its sum is $\,f(x)\,$
whenever $\,x\in V_q^+\,.$
\end{theorem}

\begin{proof}
Let $\,g(x)\,$ denote the sum of $\,S_q^{(\nu)}[f](x)\,,\,$ the $\,q$%
-Fourier series of the function $\,f\,$ given by (\ref{fourier})-(\ref{eta}%
):
\begin{equation}
\sum_{k=1}^{\infty }a_{k}^{(\nu)}\left(f\right) J_{\nu }(qj_{k\nu
}x;q^{2})=g(x)\,.  \label{Novo-2006}
\end{equation}%
Multiplying both sides of (\ref{Novo-2006}) by $\,xJ_{\nu }\big(qj_{k\nu }x;q%
\big)\,,\;k\geq 1\,,\,$ using the uniform convergence and integrating
termwise, it results, by the orthogonality relations (\ref{ort}),
\begin{equation*}
a_{k}^{(\nu)}\left( f\right)= \frac{1}{\eta_{k,\nu}}\int_{0}^{1}tg(t)J_{\nu}%
\big(qj_{k\nu}t;q^{2}\big)d_{q}t\,.
\end{equation*}
Hence,
\begin{equation*}
a_{k}^{(\nu)}(f)=a_{k}^{(\nu)}(g)\,,
\end{equation*}
so, applying in $\,L_{q}^{2}[0,1]\,$ the completeness Theorem B to the
function $\,h(x)=f(x)-g(x)\,,$ we may conclude that $\:\displaystyle h\left(
q^{k}\right)=0\,,\;k=0,1,2,\ldots\,$ which means that
\begin{equation*}
f\left(q^k\right)=g\left(q^k\right)\,,\;\: k=0,1,2,\ldots\,.
\end{equation*}
\end{proof}

\section{Examples}

We conclude with some explicit examples of uniformly convergent
Fourier-Bessel series on a $q$-linear grid.

\begin{example}
Consider $\,f(x):=x^{\nu }\,$. Using the power series expansion of $\,J_{\nu
}(x;q^{2})\,$ and the definition of the $q$-integral, a calculation shows
that %\begin{equation*}
%\int_{0}^{1}t^{\nu +1}J_{\nu }(qj_{k\nu }t;q)d_{q}t=\frac{1-q}{qj_{k\nu }}
%J_{\nu +1}(j_{k\nu }q;q)
%\end{equation*}
\begin{equation*}
\int_{0}^{1}t^{\nu +1}J_{\nu }(qj_{k\nu }t;q^{2})d_{q}t=\frac{1-q}{qj_{k\nu }%
}J_{\nu +1}(qj_{k\nu };q^{2})
\end{equation*}%
and (\ref{cfourier})-(\ref{eta}) gives
\begin{equation*}
a_{n}(x^{\nu })=-\frac{2}{q^{\nu }j_{k\nu }}\frac{1}{J_{\nu }^{\prime
}(j_{k\nu };q^{2})}\,.
\end{equation*}%
It is trivial to check that the function $\,f(x)=x^{\nu }\,$ is $\,q$-linear
H\"{o}lder of order $\,\nu \,$
%\begin{equation}\label{q-linear-Holder-Example1}
%\Big|\big(q^{n-1}\big)^{\nu}-\big(q^{n}\big)^{\nu}\Big|\leq M q^{\nu n}\:,\quad
%n=0,1,2,\ldots\,,
%\end{equation}
%with $\,M=\left|1-q^{-\nu}\right|\,$
and, if $\,\nu >1\,$,
\begin{equation*}
x^{-\frac{3}{2}}f(x)=x^{\nu -\frac{3}{2}}\in L_{q}^{2}[0,1]
\end{equation*}%
and
\begin{equation*}
\lim_{x\rightarrow 0^{+}}x^{\nu }=0\,,
\end{equation*}%
thus, by Theorem \ref{uniform-convergence}, we may conclude that the $q$%
-Fourier-Bessel series $\,S_{q}^{(\nu )}\big[x^{\nu }\big]\,$ converges
uniformly on $\,\displaystyle\,V_{q}^{+}=\left\{ q^{n}:\,n=0,1,2,\ldots
\,\right\} \,$ whenever $\,\nu >1\,,$ hence, under this restriction, by
Theorem \ref{FinalTheorem},
\begin{equation*}
x^{\nu }=-2\,q^{-\nu }\sum_{k=1}^{\infty }\frac{J_{\nu }(qj_{k\nu }x;q^{2})}{%
j_{k\nu }J_{\nu }^{\prime }(j_{k\nu };q^{2})},
\end{equation*}%
for all $\,x=q^{n}\,,$ $\,n=0,1,2,\ldots \,.$

\vspace{1em} \noindent The convergence of the expansion of $\,x^{\nu }\,$ in
the classical Fourier-Bessel series is studied in \cite[18.22]{Watson} using
contour integral methods.
\end{example}

\vspace{0.7em}

\begin{example}
Consider $\displaystyle g_{\nu,\mu}(x;q)\equiv g(x;q)\!:=\!x^{\nu} \frac{%
(x^{2}q^{2};q^{2})_{\infty}}{(x^{2}q^{2\mu-2\nu};q^{2})_{\infty}}$, with $%
\,|x|<1\,$ and $\,\mu>\nu>-\frac 1 2$.

\noindent Using the $q$-binomial theorem \cite[(1.3.2)]{GR}\ we have
\begin{equation*}
g(x;q)=\frac{x^{\nu }}{1-x^{2}}\left[\sum_{n=0}^{\infty } \frac{%
(q^{2\mu-2\nu };q^{2})_{n}}{(q^{2};q^{2})_{n}}\,x^{2n}\right]^{-1}.
\end{equation*}
Since
\begin{equation*}
\lim_{q\rightarrow 1^{-}}\sum_{n=0}^{\infty }\frac{(q^{2\mu -2\nu};q^{2})_{n}%
}{(q^{2};q^{2})_{n}}\,x^{2n}= \sum_{n=0}^{\infty }\frac{(2\mu -2\nu)_{n}}{n!}%
\,x^{2n}=(1-x^{2})^{-2\mu +2\nu }\text{,}
\end{equation*}%
it becomes clear that $\,g(x;q)\,$ is a $\,q$-analogue of $%
\,g(x)=x^{\nu}(1-x^{2})^{2\mu -2\nu -1}\,$.

We can expand $\,g(x;q)\,$ in uniform convergent $\,q$-Fourier-Bessel
series. Setting $\,x=qj_{k\nu }\,$ in formula (4.11) from \cite{A1}, it is
easy to see, following (\ref{cfourier}), that
\begin{equation*}
\int_{0}^{1}t\,g(t;q)J_{\nu }(qj_{k\nu }t;q^{2})d_{q}t=(1-q)(qj_{k\nu
})^{\nu-\mu} \frac{(q^{2};q^{2})_{\infty }}{(q^{2\mu -2\nu };q^{2})_{\infty}}%
J_{\mu }(qj_{k\nu };q^{2})\,,
\end{equation*}
therefore, (\ref{cfourier})-(\ref{eta}) enables one to write
\begin{equation*}  \label{CoeficienteExemplo2}
a_{k}^{(\nu)}\big(g(x;q)\big)=-2q^{1-\mu}\big(j_{k\nu}\big)^{\nu-\mu} \frac{%
(q^{2};q^{2})_{\infty}}{(q^{2\mu-2\nu};q^{2})_{\infty}} \frac{J_{\mu}\big(%
qj_{k\nu};q^{2}\big)}{J_{\nu+1}\big(qj_{k\nu};q^{2}\big) J_{\nu}^{\prime}%
\big(j_{k\nu};q^{2}\big)}\,.
\end{equation*}
It is easy to check that $\,g(x;q)\,$ is $\,q$-linear H\"{o}lder of order $%
\,\nu+2\,$ and, if $\,\nu>1\,$, $x^{-\frac 3 2}g(x;q)\in L_q^2[0,1]$ and $%
\lim_{x\to 0^+}g(x;q)=0$, thus, by Theorem \ref{uniform-convergence}, we may
conclude that the $q$-Fourier series $\,S_{q}^{(\nu)}\big[g(x;q)\big]\,$
converges uniformly on $\,\displaystyle
\,V_{q}^{+}=\left\{ q^{n}:\,n=0,1,2,\ldots \,\right\}\,$ whenever $\,\nu>1\,$%
, hence, by Theorem \ref{FinalTheorem},
\begin{equation*}
x^{\nu }\frac{(x^{2}q^{2};q^{2})_{\infty}}{(x^{2}q^{2(\mu-\nu)};q^{2})_{%
\infty}}= -\frac{2\,q^{1-\mu}\big(q^2;q^2\big)_{\infty}}{\big(%
q^{2(\mu-\nu)};q^2\big)_{\infty}} \sum_{k=1}^{\infty }\frac{%
(j_{k\nu})^{\nu-\mu}J_{\mu}\big(qj_{k\nu};q^{2}\big)J_{\nu}\big(%
qj_{k\nu}x;q^{2}\big)}{ J_{\nu+1}\big(qj_{k\nu};q^{2}\big)%
J_{\nu}^{\prime}(j_{k\nu };q^{2})}\,.
\end{equation*}
We note that choosing $\,\mu=\nu+1$ in the latter example one obtains the
first one.
\end{example}

\subsection*{Acknowledgements}

The authors are very grateful to Prof. Juan Arias for pointing out the Lemma %
\ref{juan}. Also, stimulating discussions with Prof. Jos\'{e} Carlos
Petronilho from the University of Coimbra are kindly acknowledged. The
research of L. D. Abreu was supported by Austrian Science Foundation (FWF)
START-project FLAME (\textquotedblleft Frames and Linear Operators for
Acoustical Modeling and Parameter Estimation\textquotedblright , Y 551-N13).
The research of R. \'{A}lvarez-Nodarse was partially supported by
MTM2015-65888-C4-1-P (Ministerio de Econom\'{\i}a y Competitividad),
FQM-262, FQM-7276 (Junta de Andaluc\'{\i}a) and Feder Funds (European
Union). The research of J. L. Cardoso was supported by the Portuguese
Government through the FCT - Funda\c{c}\~{a}o para a Ci\^{e}ncia e a
Tecnologia - under the project PEst-OE/MAT/UI4080/2014.


\bigskip

\begin{thebibliography}{99}
\bibitem{A1} L. D. Abreu, \emph{A }$q$-\emph{Sampling theorem related to the
}$q$\emph{-Hankel transform,} Proc. Amer. Math. Soc. \textbf{133}(4),
(2005), 1197-1203.

\bibitem{A2} L. D. Abreu, \emph{Functions }$q$\emph{-orthogonal with respect
to their own zeros,} Proc. Amer. Math. Soc. Volume \textbf{134}(9), (2006),
2695-2701. %, DOI S 0002-9939(06)08285-2

\bibitem{AJMAA2007} L. D. Abreu, \emph{Real Paley-Wiener theorems for the
Koornwinder-Swarttouw }$q$\emph{-Hankel transform, }J. Math. Anal. Appl.%
\textsl{,} 334 (2007), 223-231.

\bibitem{ABC} L. D. Abreu, J. Bustoz, J. L. Cardoso, \emph{The roots of the
third Jackson q-Bessel function}, Internat. J. Math. Math. Sci. \textbf{67},
(2003), 4241-4248. %, DOI 10.1155/S016117120320613X

\bibitem{AB} L. D. Abreu, J. Bustoz, \emph{On the completeness of sets of
q-Bessel functions }$J_{\nu }^{(3)}(x;q)$\emph{, }in\emph{Theory and
Applications of Special Functions. A volume dedicated to Mizan Rahman },
29--38, Dev. Math., 13, Springer, New York, 2005 (eds. M. E. H. Ismail and
H. T. Koelink).

\bibitem{AJPA} L. D. Abreu, \emph{Completeness, special functions, and
uncertainty principles over }$q$-\emph{linear grids, } J. Phys. A, Math.
Gen., \textbf{39} (2006), 14567-14580.

\bibitem{qlinearwave} L. D. Abreu, \'{O}. Ciaurri, J. L. Varona, A q-linear
analogue of the plane wave expansion, Adv. Appl. Math. 50 (2013), 415-428

\bibitem{Starlikeness2} \.{I} Akta\c{s}, \'{A} Baricz, \emph{Bounds for
radii of starlikeness of some }$q$\emph{-Bessel functions}, Result. Math.
\textbf{72}(1), (2017), 947-963.

\bibitem{Ann} M. H. Annaby, $q$\emph{-type sampling theorems, }Result.
Math., \textbf{44}(3-4), 214-225 (2003).

\bibitem{AnnMans} M. H. Annaby, Z. S. Mansour, \emph{Basic Sturm-Liouville
problems}, J. Phys. A, Math. Gen. \textbf{38}(17), (2005), 3775-3797.

\bibitem{B} M. H. Annaby, Z. S. Mansour, O. A. Ashour, \emph{Asymptotic
formulas for eigenvalues and eigenfunctions of }$q$\emph{-Sturm-Liouville
problems,} J. Phys. A. \textbf{43}(29), (2010), 295204.

\bibitem{qcalculus} MH Annaby, ZS Mansour, $q$\emph{-Fractional Calculus and
Equations, }Springer, Heidelberg (2012).

\bibitem{AAR} G. E. Andrews, R. Askey, R. Roy, \emph{Special functions.}
Encyclopedia of Mathematics and its Applications, 71. Cambridge University
Press, Cambridge, 1999

\bibitem{Aria} J. Arias de Reyna, \emph{Pointwise convergence of Fourier
Series}, Springer. Lect. Notes Math. 1785, (2000).

\bibitem{Starlikeness1} A. Baricz, D. K. Dimitrov, I. Mez\H{o}, \emph{Radii
of starlikeness and convexity of some }$q$\emph{-Bessel functions}, J. Math.
Anal. Appl. \textbf{435}, (2016), 968-985.

\bibitem{BBEB} HN. Bettaibi, N. Bouzeffour, H.B. Elmonser and W. Binous,
\emph{Elements of harmonic analysis related to the third basic zero order
Bessel function}, J. Math. Anal. Appl., \textbf{342}, (2008), 1203-1219.

\bibitem{BH} W. Bergweiler,\ W. K. Hayman, \emph{Zeros of solutions of a
functional equation}, Comput. Methods Funct. Theory, \textbf{3} (2003),
55-78.

\bibitem{BS} J. Bustoz, S. K. Suslov, \emph{Basic Analog of Fourier Series
on a }$q$\emph{-quadratic grid}, Meth. Appl. Anal. \textbf{5}, (1998), 1-48.

\bibitem{BC} J. Bustoz, J. L. Cardoso, \emph{Basic Analog of Fourier Series
on a }$q$\emph{-linear grid}, J. Approx. Theory, \textbf{112}, 134-157
(2001) \

\bibitem{JLC} J. L. Cardoso, \emph{Basic Fourier series on a }$q$\emph{%
-linear grid: Convergence theorems}, J. Math. Anal. Appl., \textbf{323}
(2006) 313-330

\bibitem{JLC2} J. L. Cardoso, \emph{Basic Fourier series: convergence on and
outside the q-linear grid}, J. Fourier Anal. Appl. (2011) \textbf{17}:
96-114.

\bibitem{JLC3} J. L. Cardoso, \emph{A few properties of the Third Jackson
q-Bessel Function}, Analysis Mathematica, \textbf{42:4} (2016)

\bibitem{JLC4} J. L. Cardoso, \emph{On basic Fourier-Bessel expansions},
(submitted), http://arxiv.org/abs/1707.05216

\bibitem{CP} J. L. Cardoso, J. C. Petronilho, \emph{Variations around
Jackson's Quantum Operator}, Methods and Applications of Analysis, \textbf{22%
}(4), (2015), 343-358.

\bibitem{Daalhuis} A. B. O. Daalhuis, \emph{Asymptotic expansions for }$q$%
\emph{-Gamma, }$q$\emph{-exponential and }$q$\emph{-Bessel functions},\emph{%
\ }J. Math. Anal. Appl., \textbf{186}, (1994), 896-913.

\bibitem{E} H. Exton, \emph{$q$-Hypergeometric Functions and Applications}, {%
\emph{J}ohn Wiley \& Sons}, (1983).

\bibitem{GR} G. Gasper, M. Rahman, \emph{Basic Hypergeometric Series},
Cambridge Univ. Press, Cambridge, UK, 1990.

\bibitem{Hay} Hayman WK, \emph{\ On the zeros of a }$q$\emph{-Bessel function%
}, Complex Analysis and Dynamical Systems II, \textbf{382}, (2005), 205-216.

\bibitem{Hardy} G. H. Hardy, \emph{Notes on special systems of orthogonal
functions (II): On functions orthogonal with respect to their own zeros}, J.
Lond. Math. Soc. \textbf{14}, (1939), 37-44.

\bibitem{I1} M. E. H. Ismail, \emph{The Zeros of Basic Bessel functions, the
functions }$\emph{J}_{v+ax}(x)$\emph{\ , and associated orthogonal
polynomials}, J. Math. Anal. Appl. \textbf{86}, (1982), 1-19.

\bibitem{Ismail} M. E. H. Ismail, \emph{Properties of the third Jackson }$q$%
\emph{-Bessel function}, unpublished manuscript.\emph{\ }

\bibitem{IsmailZAA} M. E. H. Ismail, \emph{Orthogonality and completeness of
q-Fourier type systems.} Z. Anal. Anwendungen, 20 (2001), no. 3, 761--775.

\bibitem{Ism} M. E. H. Ismail, \textit{Classical and quantum orthogonal
polynomials in one variable}, Encyclopedia of Mathematics and its
Applications No.~98, Cambridge University Press, 2005.

\bibitem{Koelink} H. T. Koelink, \emph{The quantum group of plane motions
and the Hahn-Exton }$\emph{q}$\emph{-Bessel function, }Duke Math. J. 76 (2),
483-508.

\bibitem{KS} H. T. Koelink, R. F. Swartouw,\ \emph{On the zeros of the
Hahn-Exton q-Bessel Function} \emph{and associated q-Lommel polynomials}, J.
Math. Anal. Appl. \textbf{186}, (1994), 690-710.

\bibitem{KoorS} T. H.\ Koornwinder, R. F. Swarttouw, \emph{On q-analogues of
the Fourier and Hankel transforms}, Trans. Amer. Math. Soc. \textbf{333}(1),
(1992), 445--461.

\bibitem{Littlewood} J. E. Littlewood, \emph{On the Asymptotic Approximation
to Integral Functions of Zero Order}, Proc. Lond. Math. Soc., \textbf{s2-5}%
(1), (1907), 361--410.

\bibitem{SS2016} F.\v{S}tampach, P.\v{S}\v{t}ov\'{\i}\v{c}ek,\ \emph{The
Nevanlinna parametrization for }$q$\emph{-Lommel polynomials in the
indeterminate case}, J. Approx. Theory, \textbf{201}, 2016, 48-72.

\bibitem{Watson} G. N. Watson, \emph{A treatise on the Theory of Bessel
Functions,} second edition, Cambridge University Press, Cambridge, 1966.
\end{thebibliography}
\end{document}